\documentclass{amsart}

\usepackage[utf8]{inputenc}
\usepackage{amsfonts, amsthm, amssymb, mathtools}
\usepackage{blindtext}
\usepackage{hyperref}
\usepackage{tikz,tikz-cd}
\usepackage{array}
\usepackage[style=alphabetic,sorting=nyt,maxbibnames=99]{biblatex}
\renewbibmacro{in:}{}
\tikzset{pullback/.style={minimum size=1.2ex,path picture={
\draw[opacity=1,black,-,#1] (-0.5ex,-0.5ex) -- (0.5ex,-0.5ex) -- (0.5ex,0.5ex);%
}}}

\theoremstyle{plain}
\newtheorem{theorem}{Theorem}[section]
\newtheorem*{theorem*}{Theorem}
\newtheorem{proposition}[theorem]{Proposition}
\newtheorem{lemma}[theorem]{Lemma}
\newtheorem{corollary}[theorem]{Corollary}

\theoremstyle{definition}
\newtheorem{example}[theorem]{Example}
\newtheorem{remark}[theorem]{Remark}
\newtheorem{definition}[theorem]{Definition}

\DeclarePairedDelimiter\gen{\langle}{\rangle}
\newcommand{\Set}{\mathrm{Set}}
\newcommand{\meet}{\land}
\newcommand{\op}{\mathrm{op}}
\newcommand{\C}{\mathcal{C}}
\newcommand{\D}{\mathcal{D}}
\newcommand{\E}{\mathcal{E}}
\newcommand{\F}{\mathcal{F}}
\newcommand{\Z}{\mathbb{Z}}
\newcommand{\true}{\mathrm{true}}
\newcommand{\1}{\mathbf{1}}
\newcommand{\2}{\mathbf{2}}
\newcommand{\dq}[1]{``#1"}
\newcommand{\mo}[1]{{#1}_{\scalebox{0.6}{\text{mono}}}}
\newcommand{\ch}[1]{\chi_{#1}}
\newcommand{\ob}[1]{\mathrm{ob}(#1)}
\newcommand{\SubGrp}[1]{\mathrm{Sub}_{\mathrm{Group}}(#1)}
\newcommand{\ps}[1]{\Set^{{#1}^{\op}}}
\newcommand{\Sh}[1]{\mathrm{Sh}(#1)}
\newcommand{\id}[1]{\mathrm{id}_{#1}}
\newcommand{\mono}{rightarrowtail}
\newcommand{\defeq}{=}
\newcommand{\epi}{twoheadrightarrow}
\newcommand{\FinSet}{\mathrm{FinSet}}
\newcommand{\bL}{\text{[being a loop]}}
\newcommand{\bN}{\text{[not being a loop]}}
\newcommand{\bV}{\text{[being a vertex]}}
\newcommand{\subob}{\rightarrowtail}
\newcommand{\quo}{\twoheadrightarrow}

\newcommand{\Q}{\mathcal{Q}}
\newcommand{\Eq}{\mathrm{Eq}}
\newcommand{\comma}[2]{#1 \hspace{-1pt} \downarrow \hspace{-2pt} #2}
\newcommand{\DirGraph}{\mathrm{DirGraph}}
\newcommand{\sha}{\mathbf{a}}
\newcommand{\y}[1]{\mathrm{y}(#1)}
\newcommand{\ay}[1]{\sha\mathrm{y}(#1)}
\newcommand{\Par}{\mathrm{Par}}
\newcommand{\Lobj}{\Psi}
\newcommand{\Lmor}{\psi}
\newcommand{\Robj}{\Psi}
\newcommand{\Rmor}{\psi}
\newcommand{\excl}{!}
\newcommand{\R}{\mathbb{R}}
\newcommand{\Xiz}{\Xi_0}
\newcommand{\symG}[1]{\mathfrak{S}_{#1}}
\newcommand{\Aut}[1]{\mathrm{Aut}_{\Set} (#1)}
\newcommand{\spxi}[1]{\SubGrp{\Aut{#1}}}
\newcommand{\orb}{\mathrm{orb}}
\newcommand{\Cont}[2]{\mathrm{Cont}(#1,#2)}
\newcommand{\ceil}[1]{\lceil #1 \rceil}
\newcommand{\biangle}[1]{\langle #1 \rangle}
\newcommand{\dash}{dashrightarrow}
\newcommand{\Species}{\FinSet^{\FinSet_0}}
\newcommand{\Stab}{\mathrm{Stab}}
\newcommand{\newtopic}{\vspace{15pt}}
\DeclareMathOperator*{\colim}{colim}

\title{Internal Parameterization of Hyperconnected Quotients}
\author{Ryuya Hora}
\thanks{Graduate School of Mathematical Sciences, University of Tokyo. \url{hora@ms.u-tokyo.ac.jp}}

\begin{document}
%d%f: abstract書く．
% 途中で方針をいくらか変えているので，その整合性をはかる．
%d%読み通す/整合性/
%o%f: 引用は現在形

\begin{abstract}
One of the most fundamental facts in topos theory is the internal parameterization of subtoposes: the bijective correspondence between subtoposes and Lawvere-Tierney topologies. In this paper, we introduce a new but elementary concept, ``a local state classifier," and give an analogous internal parameterization of hyperconnected quotients (i.e., hyperconnected geometric morphisms from a topos). As a corollary, we obtain a solution to the Boolean case of the first problem of Lawvere's open problems.
\end{abstract}
\maketitle
\tableofcontents
%d%f: 定義される語をイタリックに
%%%lscという語の言い訳書く?

\section{Introduction}\label{introduction}
%%%r イントロを大幅に書き直す．
%o%r: Lawvereの権威主義的な言い方の解消
%引用は現在形
% The first problem in Lawvere's \textit{open problems in topos theory} \cite{OpenLawvere} asks whether the number of \emph{quotients} of a Grothendieck topos is small or not. Here, a quotient of a topos $\E$ is a (suitable equivalent class of) \emph{connected geometric morphism} from $\E$, i.e., a geometric morphism whose inverse image part is fully faithful.

% In \cite{OpenLawvere}, for the case where the number of quotients is small, Lawvere further requires the \emph{internal parameterization} of them as follows:
% \begin{quote}
%     \dq{Is there a Grothendieck topos for which the number of these quotients is not small? At the other extreme, could they be parameterized internally, as subtoposes are?}
% \end{quote}

%d%r: 先行研究がないことをおしてもいい．/未開の地に踏み入れたんだという，ね．
%d%r: どの重要性をpushするか?
%d%r: 先行研究についてもっと触れる．/絶対に突っ込まれる．/新規性の明確化
Lawvere listed open problems in topos theory in \cite{OpenLawvere}. The first problem is as follows:
\begin{quote}
    \dq{Is there a Grothendieck topos for which the number of these quotients is not small? At the other extreme, could they be parameterized internally, as subtoposes are?}
\end{quote}
He asks whether the number of \emph{quotients} of a Grothendieck topos is small. Here, a quotient of a topos $\E$ refers to a (suitable equivalence class of) \emph{connected geometric morphism} from $\E$, i.e., a geometric morphism whose inverse image part is fully faithful.

Furthermore, for the case where the number of quotients is small, Lawvere further requires an \emph{internal parameterization} of them. The word \dq{internal parameterization} here means a bijective correspondence between quotients and \dq{internal structures.} Recall the case of subtoposes that Lawvere mentions in the quote.
The internal parameterization of subtoposes is the bijective correspondence between subtoposes of a topos $\E$ (i.e., geometric embedding \emph{into} $\E$) and Lawvere-Tierney topologies \emph{in} $\E$ \cite[see][Theorem A4.4.8]{johnstone2002sketchesv1}. Since a Lawvere-Tierney topology is defined as an internal structure (namely, internal semilattice idempotent homomorphism on the subobject classifier), this bijective correspondence is worth being called the internal parameterization of subtoposes. Lawvere seeks a similar internal parameterization for quotient toposes.

There are several motivations for obtaining an internal parameterization of quotients. First, it makes it possible to classify all quotients just by studying a specific object in the topos without dealing with vast amounts of data about the entire category. Also, correspondence with an internal structure provides a new perspective on quotients and may lead to a new operation on the class of quotients. (As we explain more concretely in a few paragraphs, our internal parameterization for \emph{hyperconnected quotients} realizes both advantages.)

However, only some previous works pay attention to internal parameterizations of quotients. Although some papers, including \cite{freyd1980axiom}\cite{rosenthal1982quotient}\cite{el2002simultaneously}, classify some limited classes of quotients, their focuses are on something other than internal parameterization. One exception is Henry's study \cite{henry2018localic} on the localic isotropy group, which classifies all \emph{atomic quotients} using internal structures. However, the class of atomic quotients is relatively small as a subclass of quotients. (In fact, the scope of our main theorem, hyperconnected quotients, properly includes it (see Example \ref{ExampleAtomicConnected})).
% Therefore, this paper might be a novel step toward solving the open problem.
%%r: これ，唐突な文．

% The word ``internally" means that a Lawvere-Tierney topology is defined as an idempotent internal semilattice homomorphism on the subobject classifier, which has the canonical internal Heyting algebra structure. In other words, the external structure, subtopos, is parameterized by the internal structure, Lawvere-Tierney topology. Lawvere is seeking a similar internal parameterization for quotient toposes.

%  geometric embedding was important in the surj-embedding factorization. Geometric embedding is an idempotent internal semilattice endomorphism of the Lawvere-Tierney topology, i.e., an internal Heyting algebra called a subobject classifier. The geometric embedding corresponded one-to-one with the Lawvere-Tierney topology, i.e., the idempotent internal semilattice endomorphism of the internal Heyting algebra called a subobject classifier. [citation needed] 
\newtopic
The main result of this paper is giving an internal parameterization of \emph{hyperconnected quotients}. (In this paper, a hyperconnected geometric morphism from a topos $\E$ is referred to as a hyperconnected quotient of $\E$, emphasizing the aspect as a quotient of the topos $\E$.) In detail, we introduce the notion of \emph{a local state classifier} and prove the following main theorem.
\begin{theorem*}[\ref{mainTheorem}]
Let $\E$ be a topos with a local state classifier $\{\xi_{X}\colon X\to\Xi\}_{X\in \ob{\E}}$ (for example, an arbitrary Grothendieck topos). Then the following three concepts correspond bijectively.
\begin{enumerate}
    \item hyperconnected quotients of $\E$
    \item internal filters of $\Xi$
    \item internal semilattice homomorphisms $\Xi\to \Omega$
\end{enumerate}
\end{theorem*}

Our result gives a partial solution to the open problem in two ways. First, since a hyperconnected quotient is a particular case of a quotient, it is a solution for the subclass of quotients. The second, somewhat nontrivial, is that our result solves the case of Boolean toposes (Corollary \ref{mainCor}). 
For a Boolean Grothendieck topos, whose quotients are automatically hyperconnected, we establish the internal parameterization of all quotients in this paper (Corollary \ref{CorInternalparameterizationOfQuotientsOfBooleanTopos}). 
% In other words, for Boolean cases, we solved the open problem. 
Thus, this paper might be a novel step toward solving the open problem, especially from the perspective of internal parameterization.

However, a hyperconnected quotient is not just a technical assumption for partially solving the open problem but has received much attention in topos theory. It naturally arises everywhere in topos theory and plays an important theoretical role. For example, for a topological group $(G,\tau)$, its continuous action topos $\Cont{G}{\tau}$ is a hyperconnected quotient of its discrete action topos $\ps{G}$ \cite[][III.9. and VII.3.]{maclane2012sheaves}\cite[][A4.6]{johnstone2002sketchesv1} The well-founded part of a topos is also a hyperconnected quotient of the topos \cite[][section 8]{freyd1980axiom}. Topos theory has other fundamental examples of hyperconnected quotients (see subsection \ref{SubsecHyperconnected}). It also has theoretical importance. For example, it plays a central role in the factorization system of geometric morphisms called hyperconnected-localic factorization, introduced in \cite{johnstone1981factorization}. 
%o%r:もう一声あっても良いかも
% In this paper, a hyperconnected geometric morphism from a topos $\E$ is referred to as a \emph{hyperconnected quotient} of $\E$, emphasizing the aspect as a quotient of $\E$.

The attempt to describe all hyperconnected quotients is not new in itself.
In \cite{rosenthal1982quotient}, Rosenthal shows that all hyperconnected quotients of a Grothendieck topos are constructed using the data called \emph{quotient systems}.

However, our result is new in the following three respects. Our description of hyperconnected quotients is internal, applies to a broader class of toposes, and utilizes a new concept, a local state classifier. We explain each of them in detail.

The first and decisive point is that our result realizes internal parameterization. While exhausting all hyperconnected quotients, Rosenthal's result does not construct a natural bijective correspondence, nor does it use internal structures. 
% In this paper, by defining a new but elementary concept, \emph{a local state classifier}, and using its internal semilattice structure, we construct a previously unknown internal parameterization, similar to that of subtoposes. 
In this paper, we introduce the notion of a local state classifier $\Xi$ of a topos and establish a natural bijective correspondence between hyperconnected quotients, internal filters of $\Xi$, and internal semilattice homomorphisms $\Xi \to \Omega$.
It enables us to classify all hyperconnected quotients just by considering one object, a local state classifier $\Xi$, without considering the whole data of the category. For example, as we see in section \ref{SecExampleApplication}, classifications of all hyperconnected quotients of the directed graph topos, group action topos, and topos of combinatorial species \cite{joyal1981theorie} are reduced to the calculation of a specific finite graph, the lattice of subgroups of $G$, and subgroups of symmetric groups, respectively. The internal parameterization also enables us to consider relationships with other internal structures. For example, in section \ref{secConclusion}, we observe that Lawvere-Tierney topologies naturally act on the class of hyperconnected quotients. 
%d%main theorem, introに書くべきだな...

The second, relatively minor, novelty is that the scope of our theory is broader than the class of Grothendieck toposes. Although our theory does not apply to all elementary toposes, it applies to a sufficiently broad class of them, including all Grothendieck toposes. 

The third novelty is the new concept of a local state classifier itself. This concept, which plays a central role in this paper, is defined in elementary category-theoretic terms. Therefore we can consider it for general categories, not limited to toposes. It leaves room for theoretical developments in various category theories. We mention this point again at the end of this introduction.
% Although Rosenthal's result is suggestive in our context, 
% As an immediate corollary, it is easily proved that the number of hyperconnected quotients of a Grothendieck topos is small. 
% Although Rosenthal's main focus in \cite{rosenthal1982quotient} was not an internal parameterization of hyperconnected quotients,
% which is the construction of a hyperconnected geometric morphism from \'{e}tendue to a given Grothendieck topos, 

% It is a connected geometric morphism (i.e., quotient) satisfying the additional condition. More specifically, it is a connected geometric morphism such that the essential image of the inverse image functor is closed by subquotients \cite[see][A4.6]{johnstone2002sketchesv1}. 

% \begin{center}
%     \begin{tikzpicture}[scale =0.8]
%     \draw[black, thick](0,0) circle (4);
%     \draw[black, thick](0,-1.2) circle (2.5);
%     \filldraw [black] (0,1.3) circle (0pt) node[above]{hyperconnected quotient};
%     \filldraw [black] (0,4) circle (0pt) node[above]{quotient};
%     \end{tikzpicture}
% \end{center}%この図，マジでダサいけど，脳の容量の僅かな節約にはなる．最終的には消すだろうけど．

% Hyperconnected geometric morphism is part of the important hyperconnected-localic factorization by Johnstone [John], along with surjection-embedding.
\newtopic
%d%r: この第三部で何をするかもっと明確に書ける．/ 共通点と双対があり/ パラレルに考えることは指針になり，/ その中でstriking similarityにも出会う．
We develop our theory in parallel with the case of subtoposes. We encounter sometimes expected, sometimes unexpected similarities and dualities. We list those analogies in Figure \ref{PictureAnalogy}.
\begin{figure}
    \centering
    \begin{tabular}{ m{2cm}|m{4.5cm}|m{4.5cm} }
  & embedding & hyperconnected \\ 
  \hline
 the \dq{small} part of  & surjection-embedding factorization &hyperconnected-localic factorization\\
  \hline
 a (external) description&Grothendieck topology (for a presheaf topos)& Quotient System \cite{rosenthal1982quotient} (for a Grothendieck topos) \\
 \hline
 the central semilattice& $\Omega$: subobject classifier&$\Xi$: local state classifier (Definition \ref{DefLocalStateClassifier})\\
 \hline
 In a presheaf topos, & $\Omega$ consists of subobjects of representable functors (i.e., sieves) & $\Xi$ consists of co-subobjects of representable functors (Example \ref{LSCofPresheafTopos}) \\
 \hline
 internal \dq{parameter} &$\Omega\to \Omega$: idempotent semilattice homomorphism (Lawvere-Tierney topology)& $\Xi \to \Omega$: semilattice homomorphism (main theorem \ref{mainTheorem})\\
\end{tabular}
    \caption{Similarity with the internal parameterization of embeddings}
    \label{PictureAnalogy}
\end{figure}
%%r: ここ練ろう．

First, a motivating analogy is that both subtoposes and hyperconnected quotients are \dq{small} parts of factorization systems (see Figure \ref{PictureFactInt}). For an arbitrary Grothendieck topos $\E$, the number of surjective geometric morphisms from $\E$ is not necessarily small. However, that of subtoposes (i.e., geometric embedding into $\E$) is always small due to the internal parameterization. That is what we mean by the term \dq{small} part. Similarly, a hyperconnected quotient is a \dq{small} part of the hyperconnected-localic factorization. For an arbitrary Grothendieck topos $\E$, the number of localic geometric morphisms to $\E$ is not necessarily small. However, that of hyperconnected quotients (i.e., hyperconnected geometric morphisms from $\E$) is always small (which is immediately followed by Rosenthal's result \cite{rosenthal1982quotient}). 
%分け...ない！
From this point of view, it is quite natural to consider an internal parameterization of hyperconnected quotients. The above discussion shows that internal parameterization (that implies their smallness, as Lawvere mentions in the quote) is impossible for localic or surjective geometric morphisms. Therefore, the missing piece is the case of hyperconnected quotients (see Figure \ref{PictureFactInt})!
% \begin{center}
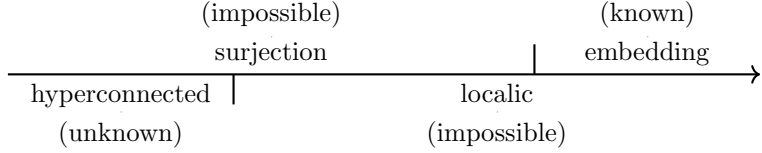
\begin{figure}
    \centering
    \begin{tikzpicture}
    \draw[black, thick, ->] (-5,0) -- (5,0);
    % \draw[black, thick] (-5,-0.4) -- (-5,0.4);
    \draw[black, thick] (-2,-0.4) -- (-2,0);
    \draw[black, thick] (2,0) -- (2,0.4);
    % \draw[black, thick] (5,-0.4) -- (5,0.4);
    \filldraw [black] (-1.5,0) circle (0pt) node[above]{surjection};
    \filldraw [black] (-1.5,0.5) circle (0pt) node[above]{(impossible)};
    \filldraw [black] (3.5,0) circle (0pt) node[above]{embedding};
    \filldraw [black] (3.5,0.5) circle (0pt) node[above]{(known)};
    \filldraw [black] (1.5,0) circle (0pt) node[below]{localic};
    \filldraw [black] (1.5,-0.5) circle (0pt) node[below]{(impossible)};
    \filldraw [black] (-3.5,0) circle (0pt) node[below]{hyperconnected};
    \filldraw [black] (-3.5,-0.5) circle (0pt) node[below]{(unknown)};
    \end{tikzpicture}
    \caption{Factorization systems and Internal parameterizations}
    \label{PictureFactInt}
\end{figure}
% \end{center}

Second, for both parameterizations, the internal correspondents are internal semilattice homomorphisms.
Recall that a Lawvere-Tierney topology is an idempotent internal semilattice homomorphism on the subobject classifier $\Omega$. The internal parameterization of subtoposes states that these idempotent internal semilattice homomorphisms bijectively correspond to subtoposes. The method of our theory is similar. First, we define an internal semilattice $\Xi$ called \emph{a local state classifier}. Then, we establish a natural bijection between hyperconnected quotients and internal semilattice homomorphisms $\Xi\to \Omega$ from the local state classifier to the subobject classifier.

%一文目だけでわかるようにしよう
Lastly, one unexpected analogy is that a local state classifier and a subobject classifier, the central semilattices in each internal parameterization, are dual in a presheaf topos! Recall that the subobject classifier of a presheaf topos is the presheaf of all subobjects of representable presheaves (i.e., sieves). By contrast, the local state classifier of a presheaf topos is the presheaf of all co-subobjects of representable presheaves (Example \ref{LSCofPresheafTopos}).

\newtopic
%d%f: would like を消す
%%r: a local state classifierが重要であることを述べるパラグラフがない．
% This paper also gives a new perspective on the quotient system given by Rosenthal \cite{rosenthal1982quotient} in terms of internal semilattices.
% Another important aspect of this paper is that it presents a seemingly too na\"ive concept, a local state classifier, with a notable topos-theoretic application. We conclude this introduction with several points about the appeal of a local state classifier. 
Finally, we conclude this introduction by reiterating the appeal of a local state classifier.
% Needless to say, it is the key to realizing our internal parameterization. 
The first thing to emphasize in the context of this paper is its theoretical necessity.
% % A local state classifier plays a central role throughout this paper. 
% We first define it for a general category 
% % in subsection \ref{subsectionDefinitionIntuitionLSC}. Then, we
% and
% prove the existence of it for an arbitrary Grothendieck topos
% % in subsection \ref{subsectionLSCofGrothendiecktopos}
% . Then, we prove our main theorem for an arbitrary elementary topos with a local state classifier (of course, including all Grothendieck toposes).
As mentioned above, a local state classifier plays a central role in our theory, like a subobject classifier in the case of subtoposes.

Despite its theoretical importance, the definition of a local state classifier (Definition \ref{DefLocalStateClassifier}) is unexpectedly simple: it is a colimit of all monomorphisms! At first glance, this definition might seem odd. In section \ref{SecLSC}, we explain as intuitively as possible how studying hyperconnected quotients leads us to this simple definition and discuss its existence and properties.

%d%topic sentenceが曖昧
A local state classifier is often given by a familiar concept. 
% Some of the specific examples of local state classifiers include familiar mathematical concepts. 
% In other words, like many other concepts defined in elementary category-theoretic terms, a local state classifier is a categorical generalization of some familiar concepts. 
% This elementarily defined concept is not just a technical gadget for topos theory but a category-theoretic toy that leaves room for play! 
For example, the local state classifier of the topos of directed graphs is the bouquet with $2$ edges
\[
\begin{tikzcd}
\bullet\ar[loop left]\ar[loop right,"."]
\end{tikzcd}
\]
That of a group action topos $\ps{G}$ is the set of all subgroups $\SubGrp{G}$ of $G$, equipped with the conjugate action (Example \ref{ExampleOfLSCOfGroupActionTopos}). That of the topos of sheaves over a topological space is the terminal sheaf (Example \ref{ExampleLSCofSheaves}). 
% That of the topos of all combinatorial species \cite{joyal1981theorie} is \dq{the species of symmetries} (Example \ref{exampleSpecies}).
See subsection \ref{subsectionOfExampleOfLSC} for these and other examples. These phenomena enable us to connect the classifications of hyperconnected quotients and existing mathematical concepts (see section \ref{SecExampleApplication}). 
% In subsection \ref{subsectionOfExampleOfLSC}, we emphasize this point by listing other examples. 

%d%概要書く
\subsection*{Overview}
In section \ref{SecTAP}, we briefly recall the notion of hyperconnected geometric morphisms and clarify some terminologies used in this paper.
In section \ref{SecLSC}, we explain some topics on a local state classifier. First, we leisurely introduce and define a local state classifier. After that, we give several examples and concretely construct a local state classifier of a Grothendieck topos. Last, we prove that a local state classifier of a cartesian closed category has an internal semilattice structure.
In section \ref{SecIPH}, we prove our main theorem, and in section \ref{SecExampleApplication}, we list examples and corollaries of it.
In section \ref{secConclusion}, we summarize what we did and list some possible future works.

In appendix \ref{AppendixInternalSemilattice}, we list the definitions and facts on internal semilattices and their filters used in our paper.
In appendix \ref{AppendixExistenceTheorem}, we give a sufficient condition for a category to have a local state classifier, which is not logically necessary for our main theorem.
% the bijective correspondence between hyperconnected quotients and internal semilattice homomorphisms $\Xi \to \Omega$. Before that, we see a broader correspondence, as a preparation.

%d%書いてから書く．

\section{Terminology and preliminaries}\label{SecTAP}
\subsection{Hyperconnected geometric morphism}\label{SubsecHyperconnected}
In this subsection, we briefly recall the definition and some properties of a hyperconnected geometric morphism. 

For the following definitions and their equivalence, see \cite[][Proposition A4.6.6]{johnstone2002sketchesv1}.
\begin{definition}[Hyperconnected geometric morphism]\label{DefOfhyperconn}
\emph{A hyperconnected geometric morphism} from $\E$ to $\F$ is a geometric morphism $f\colon \E\to \F$ that satisfies the following equivalent conditions.
\begin{enumerate}
    \item $f^*$ is full and faithful, and its essential image is closed under subobjects in $\F$
    \item $f^*$ is full and faithful, and its essential image is closed under quotients in $\F$
    \item The unit and counit of $f^* \dashv f_*$ are both monic.
    \item $f_{\ast}$ preserves the subobject classifier $\Omega$. \label{hqomega}
\end{enumerate}
\end{definition}

% A \emph{hyperconnected quotient} of a given topos $\E$ is a hyperconnected geometric morphism from $\E$. However, when we mention the number or bijective correspondence of hyperconnected quotients, they are appropriately identified, just like in the case of geometric embedding and subtoposes. There are several ways to express this identification explicitly. For example, two hyperconnected geometric morphisms are identified as a hyperconnected quotient if and only if those inverse image functors have the same essential image. From this point of view, we can think of a hyperconnected quotient as a replete coreflective full subcategory that is itself a topos and that the adjunction is hyperconnected.

A \emph{hyperconnected quotient} of a topos $\E$ is a hyperconnected geometric morphism from $\E$. We prefer the term hyperconnected \dq{quotient} rather than hyperconnected geometric morphism here to emphasize the aspect as a quotient topos, in the sense of Lawvere's open problem \cite{OpenLawvere}.

Strictly speaking, when we refer to hyperconnected quotients, we mainly refer to the equivalence classes of them based on the standard identification.
% when we later talk about the correspondence between HQ and another concept, HQ is equated in the standard way. 
This is the same situation as the famous theorem of correspondence between subtoposes and Lawvere-Tierney topologies. The standard identification can be described in several ways. One is the $2$-categorical way. Two hyperconnected geometric morphisms $f\colon \E \to \F$ and $f'\colon \E \to \F'$ are identified when there exists an equivalence $e \colon \F \to \F'$ such that the following are commutative up to a natural isomorphism
\[
\begin{tikzcd}
&&\F\ar[dd,"e", ""'{name = A}]\\
\E\ar[rru,"f"]\ar[rrd,"{f'}"']\ar[phantom, to = A,"\cong", pos =0.65]&&\\
&&{{\F}'}.
\end{tikzcd}
\]
The other is a more elementary way of looking at the inverse image functor of a hyperconnected geometric morphism. By standard identification, we identifies two hyperconnected geometric morphisms if and only if the essential images of the two inverse image functors are the same. 
% From the latter point of view, a hyperconnected quotient can be thought of as an replete coreflective subcategory satisfying additional conditions. 
%o%r: この文いらんくね

% Since the inverse image functor of a hyperconnected geometric morphism $f\colon \E \to \F$ is full and faithful, $\F$ could be regarded as a full subcategory of $\E$, up to equivalence. To avoid technical confusion caused by identification, we define hyperconnected quotients in this paper as a complete system of representatives of appropriate equivalence relation.

% \begin{definition}[hyperconnected quotient]
% Let $\E$ be an elementary topos. Then its hyperconnected quotient is a replete full subcategory $\E \fsub \Q$, such that its inclusion has the right adjoint functor and that adjunction is a hyperconnected geometric morphism.
% \end{definition}

We give several examples of hyperconnected geometric morphisms.
% \begin{definition}\label{EquivOfHyperquotients}\end{definition}

\begin{example}[Full and bijective on objects functor]\label{HQfromfbo}
If a functor $F\colon \C \to \D$ between small categories is full and bijective on objects, then the induced geometric morphism $\ps{\C}\to \ps{\D}$ is hyperconnected. For detail, see \cite[][Example A4.6.9]{johnstone2002sketchesv1}.
\end{example}

\begin{example}[Topological monoid action topos]\label{HQfromTopMonoid}
Let $(M,\tau)$ be a topological monoid and $M$ be its underlying discrete monoid. Then, the topos of continuous action $\Cont{M}{\tau}$ is a hyperconnected quotient of the presheaf topos $\ps{M}$.
Properties of topological monoid action toposes, including this hyperconnected quotient, are extensively studied by Morgan Rogers in \cite{rogers2021toposesM}\cite{rogers2021toposesT}.
%o%f: 二つのせるか?

For the case where $M$ is a group, $\Cont{M}{\tau}$ is a topological group action topos, which is explained in \cite[][III.9. and VII.3.]{maclane2012sheaves} and mentioned in \cite[][A4.6]{johnstone2002sketchesv1}.
\end{example}

\begin{example}[Relativized two-valuedness]
A hyperconnected geometric morphism is a \dq{relativized two-valuedness.} In other words, the unique geometric morphism $\excl\colon \E \to \Set$ from a Grothendieck topos $\E$ is hyperconnected if and only if $\E$ is a \emph{two-valued topos}. It is followed by condition \ref{hqomega} in Definition \ref{DefOfhyperconn}. For example, the unique geometric morphism from a presheaf topos $\excl\colon\ps{\C} \to \Set$ is hyperconnected if and only if $\C$ is strongly connected, i.e., for any ordered pair of objects $(a,b) \in \ob{\C}^{2}$, there exists at least one morphism $a\to b$ in $\C$. For detail, see \cite[][Example A4.6.9]{johnstone2002sketchesv1}.

% This is an example of the idea of \emph{relative point of view}\cite{nlab_relative_point_of_view}.
%o%r: 上の文を入れるかを迷う．
As an aside, a connected geometric morphism, which appears in Lawvere's original problem \cite{OpenLawvere}, gives a closely related example. A Grothendieck topos $\E$ is a \emph{connected topos} if and only if the unique geometric morphism $\excl\colon \E \to \Set$ is a \emph{connected geometric morphism}. See \cite[][Exercises 4.8.]{johnstone2014topos} for the precise statement, and see \cite[][Lemma C1.5.7.]{johnstone2002sketchesv2} for geometric intuition.
%o%r: 後半いらんかもね
\end{example}

\begin{example}[Localic topos]\label{HQoflocalictopoi}
Another class of toposes, \emph{localic toposes} can also be characterized in terms of hyperconnected quotients. A Grothendieck topos $\E$ is localic if and only if $\E$ itself is the only hyperconnected quotient that $\E$ has. It is followed by the hyperconnected-localic factorization \cite{johnstone1981factorization}\cite[][A4.6]{johnstone2002sketchesv1}. Localic Grothendieck toposes give a theoretically important example of our theorem. See subsection \ref{subsectionLSCinLocalic}.
\end{example}

%o%m: \begin{example}[locally decidable]\end{example}やisotoropyを 書くかも含めて悩んで

\begin{example}[Atomic quotients and well-founded part]\label{ExampleAtomicConnected} 
Atomic quotients (i.e., connected \emph{atomic} geometric morphisms) are examples of hyperconnected quotients \cite[][Lemma C3.5.4.]{johnstone2002sketchesv2}. 

In \cite[][section 8]{freyd1980axiom}, atomic quotients of a Grothendieck topos 
%o%m: 本当はcomplete toposのfull subのこと．Grothendieckなら一致する
are called \emph{exponential varieties}. The minimal exponential variety is called \emph{well-founded part} and is of particular interest in relation to set theory.

Recently, atomic quotients have been studied from another point of view, the localic isotropy group\cite{henry2018localic}.
\end{example}

\section{Local state classifier}\label{SecLSC}
This section is dedicated to the explanations of several topics of a local state classifier $\Xi$ of a category, which plays a central role in this paper.

In the first two subsections, we leisurely introduce the notion of a local state classifier. In subsection \ref{subsectionNecessityandInevita}, we observe one property of hyperconnected quotients and explain how it leads us to the definition of a local state classifier. In subsection \ref{subsectionDefinitionIntuitionLSC}, we give a formal definition and informal explanation of a local state classifier.

In the following two subsections, we give examples of it. The purpose of section \ref{subsectionOfExampleOfLSC} is just to list examples of a local state classifier and familiarize readers with them. In subsection \ref{subsectionLSCofGrothendiecktopos}, we give a concrete construction of a local state classifier of an arbitrary Grothendieck topos. It is not just an example, but a theoretical step to ensure that all Grothendieck toposes are in the scope of our main theorem (Theorem \ref{mainTheorem}).

In the last subsection \ref{subsectionSemilatticeStructure}, we prove that a local state classifier of a cartesian closed category (for example, topos) has an internal semilattice structure. As we explained in Introduction (section \ref{introduction}), our internal parameterization of hyperconnected quotients takes advantage of this internal semilattice structure. 

%o%r: Remark on the word "local state classifier"
\subsection{Necessity and inevitability}\label{subsectionNecessityandInevita}

% \subsection{Necessity and intuition of a local state classifier}\label{SectionInformalIdea}
%ここで気持ちをもっと書く．この論文の内容は，非常にシンプルな一つのアイデアに基づいている．
The whole contents of this paper, including the definition of a local state classifier and the proof of the main theorem, are led by one simple idea: \dq{hyperconnected quotients are determined by local states.} Just to clarify the meaning of this idea, we formulate it as a lemma and give proof, although it is not logically necessary to prove the main theorem. Let $\E$ be a topos, and $\Q$ be its hyperconnected quotient. Only in this subsection, we say an object $X\in \ob{\E}$ is \emph{covered} by its family of subobjects $\{U_{\lambda} \subob X\}_{\lambda \in \Lambda}$, 
%if $X$ can be realized as a colimit of $\{U_{\lambda}\}_{\lambda \in \Lambda}$. 
if the canonical morphism 
\[\coprod_{\lambda \in \Lambda}U_{\lambda} \to X\]
is epic.
Then the following holds. (Note that we can regard a hyperconnected quotient as a replete full subcategory of a topos.)
\begin{lemma}[Hyperconnected quotients are determined by local states.]
If an object $X$ is covered by $\{U_{\lambda} \subob X\}_{\lambda \in \Lambda}$, then $X$ belongs to $\Q$ if and only if all of $\{U_{\lambda}\}_{\lambda \in \Lambda}$ belong to $\Q$.
\end{lemma}
\begin{proof}
This is immediately followed by the fact that hyperconnected quotient $\Q$ regarded as a replete full subcategory of $\E$ is closed under subquotients and coproducts in $\E$.
\end{proof}
In this sense, whether or not an object $X$ belongs to $\Q$ can be determined \emph{locally}. Although it is not clear what the word \dq{local} mathematically means here, this observation lead us to imagine the description of a hyperconnected quotient $\Q$ as a collection of all \dq{local states} that belong to $\Q$. For example, if one could define the set of all local states, it would be able to state that all hyperconnected quotients are constructed by \dq{good} subset of it. Rosenthal's quotient system \cite{rosenthal1982quotient} can be regarded as one (external) realization of this idea.

%o%r: 絵をかく．

%1:Hyperconnected quotientは，quotientであってsubobjectを取る操作で閉じているもの．だから，objectがhyperconnected quotientに入っているかは"localに"判別できる！詳しくいうと，object Xがhyperqupotientに入っているかを知りたければ，そのsubobjectによる被覆を見ればいい．つまり，{S\to X}らがあれば，Xが入っていればSも全て入っているし，Sが全て入っていればcolimでSを張り合わせてXを復元すればXも入っていることがわかる． このように，objectがhyperconnected quotientに入っているかは"localに"判別できるはず．informalには，hyperconnected quotientに入っているようなlocal state のリストが与えられれば，hyperconnected quotientを復元できる．hyperconnected quotientと，local statesのリストのうち良いもの，の間には一対一対応があるだろう．(このアイデア自体ははRosenthalも考えていたと思う．Quotient Systemはその一つの実現であると捉えられる．)

%2結果的には，本論文ではこのアイデアを定式化し，local stateを全て集めてできる特別な対象，local state classifierを定義した．そして，hyperconnected quotientとlocal stateのリスト i.e., LSCの良いsubobjectとの対応を証明した．toposとそのhyperconnected quotientの包含関係を，local stateのみに注目することで，lscとその部分対象の関係に圧縮したのである!
However, our focus is not just a description of hyperconnected quotients, but an internal parameterization of them. In this paper, instead of defining the \emph{set} of all local states and considering a good subset of it, we define a special \emph{object} $\Xi$ in the considered topos $\E$, which can be regarded as a collection of all local states, and consider a good subobject of it. This object $\Xi$ is the main content of this section, a local state classifier, and plays a central role throughout this paper. (Later, the \dq{good subobject} turns out to be its internal filter.)

\subsection{Definition and intuition}\label{subsectionDefinitionIntuitionLSC}

%3:しかし，そもそもlocal stateとは何で，どう定式化すればいいのだろうか?そして，良い部分対象とは具体的には何になるのだろうか? 
In subsection \ref{subsectionNecessityandInevita}, an object $\Xi$ that realizes the slogan \dq{a collection of all local states} is required. But how could we define such an object? To answer this seemingly too abstract question, we first extend the scope of our thinking from toposes to general categories and consider a simple formulation of what \dq{local} means (at least in this context), in elementary terms. 

Before explaining the informal idea of our formulation, we first define \emph{locally determined cocone} of a category. Let $\mo{\C}$ denote the subcategory of a category $\C$ that consists of all monomorphisms and the same objects of $\C$. 

\begin{definition}[Locally determined cocone]\label{DefOfLocallyDeterminedCocone}
Let $\C$ be a category. \emph{A locally determined cocone} of $\C$ is a cocone
% \[\{\Lmor_{X}\colon X\to \Lobj\}_{X\in \ob{\C}}\] of
under the inclusion functor $\mo{\C}\to \C$ regarded as a possibly large diagram.
\end{definition}

This definition needs more explanations since it might be rare for some readers to consider such a large cocone diagram. A locally determined cocone of a category $\C$ is an object $\Psi$ equipped with a family of morphisms $\{\Lmor_X \colon X\to \Lobj\}_{X \in \ob{\C}}$ from all objects of $\C$, such that for any monomorphism $\iota \colon U\subob X$, the following diagram
\[
\begin{tikzcd}[column sep =tiny]
U\ar[rr,"\iota",\mono]\ar[rd,"\Lmor_{U}"']&&X\ar[ld,"\Lmor_{X}"]\\
&\Lobj&\\
\end{tikzcd}
\]
commutes. Informally speaking, this commutativity asserts that the value of the morphism $\Lmor_X$ is locally determined, as the name asserts. We explain what this means, using a metaphor with elements (which does make sense by considering generalized elements) and Figure \ref{PictureLocallyDeterminedCocone}. Take an element $x$ in $X$ and try to compute $\Lmor_X (x)$. If we take a subobject $U$ small enough but containing $x$, then the commutativity implies that $\Lmor _X (x)$ is equal to $\Lmor_U (x)$. In other words, in a locally determined cocone, the value $\Lmor_X (x)$ can be computed in an arbitrarily small neighborhood of $x$. 
% For a family of morphisms $\{\Lmor_X\}$ to be locally determined, the value $\Lmor_X (x)$ can not depend on \dq{global data}. 
This is the intuition behind the terminology \dq{locally determined.} Therefore, a vague word \dq{local} basically means \dq{in an arbitrary subobject} in our context.
\begin{center}
\begin{figure}
    \centering
    \begin{tikzpicture}
    %U and its inclusion into X
    \draw [black,thick] (-2,3.3) circle (0.5);
    \draw (-2,3.3+0.5)circle(0) node[above]{$U$};
    \draw [black, dotted, thick] (1.5,3.3) circle (0.5);
    \filldraw [black] (-2,3.3) circle(0.04);
    \draw (-2,3.3)circle(0) node[above]{$x$};
    
    % X and its "element."
    \draw [black, thick] (2,3) ellipse (2.5 and 1.5);
    \draw (2,3+1.5)circle(0) node[above]{$X$};
    \filldraw [black] (1.5,3.3) circle(0.04);
    \draw (1.5,3.3)circle(0) node[above]{$x$};

    \draw [black,thick,>->] (-1.5+0.1,3.3) -- (1-0.1,3.3);
    \draw (-0.8,3.3)circle(0) node[above]{$\iota$};
    \draw [black,thick,->] (-2,3.3-0.6) -- (-0.2,0.1-1);
    \draw (-1-0.5,1.2)circle(0) node[left]{$\Lmor_U$};
    \draw [black,thick,->] (2,3-1.6) -- (0.2,0.1-1);
    \draw (1+0.5,0.3)circle(0) node[right]{$\Lmor_X$};
    
    %L
    \draw [black,thick](-2,0-1.5) -- (2,0-1.5) -- (2,-1-1.5) -- (-2,-1-1.5) -- cycle;
    \draw (0,-2.5)circle(0) node[below]{$\Lobj$};
    % \filldraw [black] (-0.2,-1.7) circle(0.04);
    \draw (0,-2)circle(0) node[]{$\Lmor_U (x) = \Lmor_X (x)$};
    % \draw[black,thick](0,0) .. controls (1,1) and (2,-1) .. (3,0) ;

    \end{tikzpicture}
    \caption{Locally determined cocones}
    \label{PictureLocallyDeterminedCocone}
\end{figure}
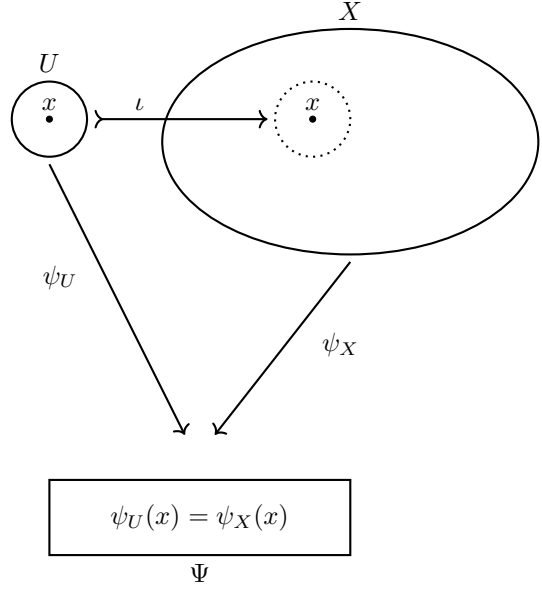
\end{center}
%4local state classifierの定義は非常にシンプルである．まずは， locally determined coconeを定義し，~~

Finally, a local state classifier $\Xi$ is defined as the universal locally determined cocone, i.e., a colimit of the diagram $\mo{\C}\to \C$.
\begin{definition}[Local state classifier]\label{DefLocalStateClassifier}
\emph{A local state classifier} $\Xi$ of a category $\C$ is a colimit of the inclusion functor $\mo{\C}\to \C$, if it exists.
The associated cocone is referred to as $\{\xi_{X}\colon X\to \Xi\}_{X\in \ob{\C}}$.
\end{definition}

\begin{example}[Toy example: $\Set$]\label{ExampleLSCofSets}
%d%書く？
Just to confirm the reader's formal understanding of the definition, one fundamental example, the category of sets $\Set$, is given in advance of the next section. 
Since $\mo{\Set}$ is the category of sets and injections $\Set_{\text{inj}}$, we think of a cocone and a colimit for the inclusion functor $\Set_{\text{inj}}\to \Set$. For any set $L$ and a picked element $l\in L$, the family of functions 
\[\{\Lmor_X \colon X \to L\}_{X\text{: set}},\]
 where each function $\Lmor_X$ sends every element of $X$ to $l$, is a locally determined cocone.
Furthermore, every locally determined cocone of the category of sets $\Set$ is in this form. (One can prove this fact by defining $l\in L$ to be $l \defeq \Lmor_{\{\ast\}}(\ast)$, for a given locally determined cocone $\{\Lmor_X\}$) Therefore, a local state classifier of the category of sets is given by the terminal cocone.
\[\{! \colon X \to \{\ast\}\}_{X\text{: set}} .\]
\end{example}

As we write \dq{if it exists} in the above definition, a local state classifier may not exist. We make a brief remark on that, although we discuss it later.

\begin{remark}[Existence of a local state classifier]
Since a local state classifier is defined as a colimit of a possibly large diagram, a local state classifier may not exist, even for a cocomplete category. However, for a Grothendieck topos, we prove the existence of it by concretely constructing it in subsection \ref{subsectionLSCofGrothendiecktopos}. A more generalized existence theorem is shown in Appendix \ref{AppendixExistenceTheorem}. For an elementary topos without a local state classifier, see Example \ref{exampleElementaryToposWithoutLSC}.
\end{remark}

By definition, all locally determined cocones uniquely factor through a local state classifier $\Xi$. In this sense, $\Xi$ and a component of the cocone $\xi_X\colon X\to \Xi$ are expected to be a collection of all local states and a morphism classifying \dq{elements} depending on their local states, as we desired.  In the next section, we check the validity of this expectation through several examples.

\subsection{Examples}\label{subsectionOfExampleOfLSC}
In the previous subsection, we define a local state classifier led by an intuition of what \dq{local} means in our context. 
% In this subsection, we list several examples of a local state classifier and make sure that it actually works intuitively. 
In this subsection, we give several examples of local state classifiers to confirm that they do fit our intuition.
% of \dq{local} given in the previous subsection \ref{subsectionDefinitionIntuitionLSC}. 
Leaving rigorous calculations and proofs for subsection \ref{subsectionLSCofGrothendiecktopos},
%(or sometimes for readers)
we concentrate on presentations and informal descriptions of them.

%d%気持ち説明の後なので，もっと大胆に気持ちを込めて例を説明する．

\begin{example}[Directed graphs]\label{ExampleLSCofDirectedGraph}
The first example is from graph theory. Let $\DirGraph$ denote the category of all directed graphs, in other words, the presheaf topos of the parallel morphism category
\[
\begin{tikzcd}
\Par\colon & V\ar[r,bend right]\ar[r,bend left]&E.
\end{tikzcd}
\]

What is a local state classifier of this category $\DirGraph \simeq \ps{\Par}$?
% Leaving a rigorous and concrete calculation to \ref{subsectionLSCofGrothendiecktopos}, 
% we now make informal considerations in line with the intuition given in the previous subsection \ref{subsectionDefinitionIntuitionLSC}.
% What should local states in this category be? To answer this question, 
% Consider one example of a directed graph.
First, we informally consider a local state of a vertex, although it is not mathematically defined. Recall that the informal idea of a local state is data that does not change within any subobject. For example, the indegree and outdegree of a vertex cannot be a local state in our context, because they change when considered within a subobject. In fact, \dq{being a vertex} is the only local state of a vertex, because, for each vertex of a directed graph, the vertex $\bullet$ itself is a subobject. We cannot locally distinguish two vertices. On the other hand, the situation is different in the case of an edge. For example, consider the following directed graph.
\[
\begin{tikzcd}
\bullet\ar[r,"a"]&\bullet\ar[r,"b"]&\bullet\ar[loop above,"c"]
\end{tikzcd}
\]
The following three subobjects are the minimum subobjects that include $a,b,c$, respectively. 
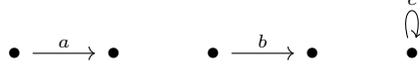
\[
\begin{tikzcd}
\bullet\ar[r,"a"]&\bullet&\bullet\ar[r,"b"]&\bullet&\bullet\ar[loop above,"c"]
\end{tikzcd}
\]
 By this observation, one might imagine that two possible local states of an edge are $\bL$ and $\bN$.

It turns out later that the local state classifier $\Xi$, calculated by Example \ref{LSCofPresheafTopos}, fits this informal observation. The local state classifier $\Xi$ has one vertex (say, $\bV$) and two edges $\bL,\bN$, as shown in the following picture.
\[
\begin{tikzcd}
\bullet\ar[loop left,"\bL"]\ar[loop right,"\bN"]%\ar[phantom,"\bV",loop below]
\end{tikzcd}
\]
The morphism $\xi_{X}\colon X\to \Xi$ for each $X \in \ob{\DirGraph}$ sends its vertex to the unique vertex of $\Xi$ and its edge $e$ to $\bL$ or $\bN$ depending on whether $e$ is a loop or not. For example, the edges $a,b$ in the following graph are sent to $\bN$, and $c$ is sent to $\bL$.
\[
\begin{tikzcd}
\bullet\ar[r,"a"]&\bullet\ar[r,"b"]&\bullet\ar[loop above,"c"]
\end{tikzcd}
\]
The local state classifier $\Xi$ does classify local states! 
%d%!は使っていいんだろうか 使うぜ！
% Informally speaking, the only local state of a vertex is just \dq{being vertex}, and the possible two local states of an edge are \dq{being loop} and \dq{not being loop}.
\end{example}

\begin{example}[Group actions] \label{ExampleOfLSCOfGroupActionTopos}
The second example is from group theory. Consider a presheaf topos $\ps{G}$ on a group $G$, i.e., the category of right $G$-sets. What is a local state of an element of a right $G$-set? To consider this question, take a right $G$-set $X$ and its element $x\in X$. One might think an appropriate notion of a local state of $x$ is the orbit $Gx$ of $x$ because the orbit is the minimum subobject of $X$ that includes $x$. It is correct, but there is a simpler description of it, using stabilizer subgroups $\Stab_{G}(x)$. (Considering orbits is the reverse side of considering stabilizer subgroups since the orbit $Gx$ and the right $G$-set of the right cosets of the stabilizer subgroup $\Stab_{G}(x)$ are isomorphic
\[
Gx \cong \Stab_{G}(x)\backslash G
\] 
as right $G$-sets.)
The local state classifier of $\ps{G}$ is a set $\Xi = \SubGrp{G}$ of all subgroups of $G$ equipped with the right conjugate action
\[
H\ast g \coloneqq g^{-1}Hg.
\]
For each right $G$-set $X$, the morphism $\xi_{X}\colon X\to \Xi$ sends an element $x\in X$ to its stabilizer subgroup \[\xi_{X} (x) = \Stab_{G}(x).\]
% Local states of group actions are stabilizer subgroups.
\end{example}

% \begin{example}[Colored sets]

% \end{example}
\begin{example}[Sheaves over a topological space]\label{ExampleLSCofSheaves}
The third example is the category of sheaves on a topological space. A local state classifier of a sheaf topos $\Sh{X}$ over a topological space $X$ is understood intuitively and, at the same time, theoretically suggestive. 

Recall that $\Sh{X}$ is equivalent to the category of \'{e}tale bundles over $X$, i.e., local(!) homeomorphisms to $X$. For any point $y \in Y$ in an \'{e}tale bundle $p\colon Y\to X$, a small enough neighborhood of $y$ is homeomorphic to that of the underlying point $p(y) \in X$. Therefore, there seems no possible local state of a point $y \in Y$, except where it locates, i.e., the underlying point $p(y)\in X$.
%there seems no possible local state of y \in Y in an etale bundle p:Y→X except where it locates, i.e., the underlying point p(x).
%In fact, for any  two points y0 and y1 in p^-1 x, there exists an open neighborhood U of x and the By considering 

As this observation suggests, the local state classifier of $\Sh{X}$, calculated by Example \ref{LSCofLocalicGrothendieckTopos}, is $X$ itself. In other words, it is the terminal \'{e}tale bundle $\id{X}\colon X\to X$.
This example includes some toy examples. First, a local state classifier of the function topos $\Set^{\to}$, which is equivalent to the sheaf topos over the Sierpi\'{n}ski space, is the terminal object. Second, a local state classifier of the topos of $3$-colored sets $\Set/\{R, G, B\}$, which is just a slice category of $\Set$ or the sheaf topos over the discrete topological space $\{R, G, B\}$, is the terminal object $\{R, G, B\}$. Unsurprisingly, the \dq{local states} of $3$-colored sets are just colors.

Later, we generalize this example to an arbitrary localic Grothendieck topos. In Example \ref{LSCofLocalicGrothendieckTopos}, we show a generalized statement that a local state classifier of a localic Grothendieck topos is the terminal object. Furthermore, in subsection \ref{subsectionLSCinLocalic}, we prove that a Grothendieck topos is localic if and only if its local state classifier is terminal, as a corollary of our main theorem.
\end{example}

%d%これを消すか悩む 消さない．
\begin{example}[Relation to a terminal object]\label{ExampleTerminalLSC}
Related to the previous example, we list several categories whose local state classifiers are terminal. Before we get into the details, notice that the definition of a local state classifier is similar to the characterization of a terminal object, which is the colimit of the identity functor $\id{\C}\colon \C\to \C$ \cite[see][Lemma 3.7.1]{riehl2017category}. If a category $\C$ has a terminal object $\1$, then $\C$ has a trivial locally determined cocone $\{\excl\colon X\to \1\}$. It is natural to ask whether this trivial cocone gives a local state classifier.
% and to regard a terminal local state classifier as a \dq{degenerate} case.

For the following categories, a local state classifier is a terminal object. 
\begin{itemize}
    \item $\Set$, the category of sets (Example \ref{ExampleLSCofSets})
    \item $\FinSet$, the category of finite sets
    \item $\mathrm{Poset}$, the category of partially ordered sets
    %o%r: \item $\mathrm{Cat}$, the category of categoriesをどっかに書く．
    \item $\mathrm{Top}$, the category of topological spaces
    \item $\mathrm{Manifold}$, the category of manifolds
    \item $\mathrm{Sh}(X)$, the category of sheaves over a topological space $X$ (Example \ref{ExampleLSCofSheaves}) %d%これが非常に直感的であることを説明する．
    \item $\mathrm{Top}/X$, the category of bundles over a topological space $X$
\end{itemize}
Those examples are immediately obtained by the general fact that if a category is cartesian and has a generating set consisting of subterminals, then its local state classifier is a terminal object. In fact, any locally determined cocone $\{\Lmor_{X} \colon X\to\Lobj\}$ is uniquely factored through $\{\excl\colon X\to \1\}$ as 
\[
\begin{tikzcd}[column sep=tiny ]
&X\ar[ld,"\excl"']\ar[rd,"\Lmor_{X}"]&\\
\1\ar[rr,"\Lmor_{\1}"']&&\Lobj.
\end{tikzcd}
\]
It is because for any subterminal object $S$ and morphism $S\subob X$,
\[
\begin{tikzcd}[column sep=tiny ]
&S\ar[d,\mono]\ar[ldd,bend right,"\excl"',\mono] \ar[rdd, bend left, "\Lmor_{S}"]&\\
&X\ar[ld,"\excl"']\ar[rd,"\Lmor_{X}"]&\\
\1\ar[rr,"\Lmor_{\1}"']&&\Lobj
\end{tikzcd}
\]
the outer triangle of the above diagram is commutative. For a relationship between this fact and localic toposes, see Example \ref{LSCofLocalicGrothendieckTopos} and subsection \ref{subsectionLSCinLocalic}.
%d%下のalgebraic structureに触れるか悩む
% \begin{itemize}
%     \item $\mathrm{Group}$, the category of groups
%     \item $\mathrm{Ab}$, the category of abelian groups
%     \item $\mathrm{Vect}_{K}$, the category of $K$-vector spaces
%     \item $\mathrm{Lattice}$, the category of lattices
%     \item $\mathrm{Ring}$, the category of rings
% \end{itemize}
\end{example}

%d%上をまとめるように，localicの話書く．

\begin{example}[Pointed sets]\label{ExampleLSCofPointedSets}
% We give an example of a category that is not a topos but has a non-terminal local state classifier. 
Some categories that are not topos also have a non-trivial (i.e., not terminal) local state classifier. One toy example is given by the category of pointed sets $\Set_{\ast}$. An object is a set with a basepoint, and a morphism between them is a function that sends a basepoint to a basepoint. The local state classifier $\Xi$ of this category $\Set_{\ast}$ is the set 
\[
\{\text{[not being a basepoint], [being a basepoint]}\}
\]
with a basepoint $\text{[being a basepoint]}$. For each pointed set $(X,x_0)$, the morphism $\xi_{(X,x_0)}\colon (X,x_0)\to \Xi$ sends the basepoint $x_0$ to $\text{[being a basepoint]}$ and others to $\text{[not being a basepoint]}$.
\end{example}

% \begin{example}[The sheaf topos over a topological space]
% Let $X$ be a topological space and $\Sh{X}$ be its sheaf topos, i.e., the category of étale bundles over $X$. The local state classifier is the terminal object $\id{X}\colon X\to X$, and $\xi$ is the unique morphism to the terminal object.
% \end{example}

%d%LSCを持つelementary toposと持たないelementary toposをそれぞれ挙げる．
Although we see later that every Grothendieck topos has a local state classifier in subsection \ref{subsectionLSCofGrothendiecktopos}, an elementary topos may or may not have a local state classifier. We give an example for each case.

\begin{example}[Combinatorial species]\label{exampleSpecies}
We give an impressive example of a local state classifier of an elementary topos from categorical combinatorics. In \cite[][1.2. Cat\'{e}gorie des esp\`{e}ces]{joyal1981theorie}, the category of species is defined as a functor category $\Species$, where $\FinSet_0$ denotes the groupoid of finite sets and bijections. This category is an elementary topos because it is equivalent to the product category of finite group action toposes
\[\Species \simeq \prod_{n=0}^{\infty} \FinSet^{\symG{n}},\]
and a product category of elementary toposes is also an elementary topos \cite[][V. Exercises 8]{maclane2012sheaves}.

A local state classifier of the category of species is a functor (or a species) $\Xi\colon \FinSet_0 \to \FinSet$ that sends a finite set $A$ to the set of all subgroups of the permutation group on $A$
\[\Xi (A) = \spxi{A}.\]
An action of a bijection $\sigma\colon A \to B$ on an element $S \in \Xi (B)$ is given by the left conjugate action $S \mapsto \sigma S \sigma^{-1}$.
For a species $M\colon \FinSet_0 \to \FinSet$, a component of cocone $(\xi_M)_A \colon M(A) \to \Xi (A)$ sends an element $s \in M(A)$ (which is called $M$-structure $s$ on $A$) to its \emph{group of symmetries} 
\[
(\xi_M)_A (s) = \{\sigma \colon A \to A \text{: bijection}\mid \sigma s =s\},
\]
where $\sigma s$ denotes the action of $\sigma$ on $s$.
For example, consider a species of undirected graphs $G\colon \FinSet_0 \to \FinSet$, which sends a finite set $A$ to the set of all undirected graphs whose underlying sets are $A$. Then, the following graphs 
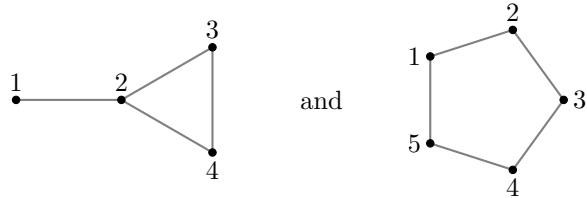
\[
\begin{tikzpicture}[scale = 0.7]
\draw[gray, thick] (0-1,0) -- (2-1,0) -- (3.73-1,1) -- (3.73-1,-1) -- (2-1,0);
\draw[gray, thick] (1*1.4+5+2+1,0) -- (0.31*1.4+5+2+1,0.95*1.4) -- (-0.81*1.4+5+2+1, 0.59*1.4) -- (-0.81*1.4+5+2+1, -0.59*1.4) -- (0.31*1.4+5+2+1,-0.95*1.4) -- (1*1.4+5+2+1,0);
\filldraw[black] (0-1,0) circle (2pt) node[above]{1};
\filldraw[black] (2-1,0) circle (2pt) node[above]{2};
\filldraw[black] (3.73-1,1) circle (2pt) node[above]{3};
\filldraw[black] (3.73-1,-1) circle (2pt) node[below]{4};

\filldraw[black] (4.8,0) circle (0pt) node{\text{and}};

\filldraw[black] (1*1.4+5+2+1,0) circle (2pt) node[right]{3};
\filldraw[black] (0.31*1.4+5+2+1,0.95*1.4) circle (2pt) node[above]{2};
\filldraw[black] (-0.81*1.4+5+2+1, 0.59*1.4) circle (2pt) node[left]{1};
\filldraw[black] (-0.81*1.4+5+2+1, -0.59*1.4) circle (2pt) node[left]{5};
\filldraw[black] (0.31*1.4+5+2+1,-0.95*1.4) circle (2pt) node[below]{4};
\end{tikzpicture}
\]
are sent by $\xi_G$ to their automorphism groups, $\Z/2\Z \subset \Aut{\{1,2,3,4\}}$ and $D_5 \subset \Aut{\{1,2,3,4,5\}}$ respectively. In this sense, a local state classifier $\Xi$ of the category of species is \dq{the species of symmetries,} which consists of all finite groups due to Cayley's theorem!
%o%m: Frucht's theorem をね．いや，頂点集合考えるとepiじゃなくね??考え直し．\xiがepiになる自然なspeciesってなんだろ...
\end{example}

\begin{example}[Elementary topos without a local state classifier]\label{exampleElementaryToposWithoutLSC}
There is an elementary topos that does not have a local state classifier. Let $G$ be a group with infinitely many finite index subgroups, such as $\Z$, and $\FinSet^{G^{\op}}$ be the topos of finite right $G$-sets. Then $\FinSet^{G^{\op}}$ does not have a local state classifier. (Compare this counterexample with Example \ref{ExampleOfLSCOfGroupActionTopos})
\end{example}

\subsection{Local state classifier of a Grothendieck topos}\label{subsectionLSCofGrothendiecktopos}
In this section, we concretely construct a local state classifier of a Grothendieck topos. Recall that even the existence of it is non-trivial, even though a Grothendieck topos is cocomplete.

We make use of the adjunction between the inclusion functor and the sheafification functor
\[
\begin{tikzcd}
\Sh{\C,J}\ar[r, hookrightarrow]&\ps{\C}\ar[l,"\sha"',bend right],
\end{tikzcd}
\]
for our construction of a local state classifier of a Grothendieck topos $\Sh{\C,J}$.
In this section, to avoid confusion about whether we are considering a diagram in the sheaf topos $\Sh{\C,J}$ or the presheaf topos $\ps{\C}$, we explicitly write next to a diagram which topos we are considering at that time. Especially, while dealing with epimorphisms, one should be more careful at this point, because an epimorphism in the sheaf topos is not necessarily epic in the presheaf topos.

Since a Grothendieck topos $\Sh{\C,J}$ is a reflective full subcategory of a presheaf topos $\ps{\C}$, once we have a colimit $\Xi_0$ of the following large diagram
\[
\begin{tikzcd}
\mo{\Sh{\C,J}}\ar[r,\mono]&\Sh{\C,J}\ar[r,hookrightarrow]&\ps{\C},
\end{tikzcd}
\]
we obtain a local state classifier of a considered Grothendieck topos $\Sh{\C,J}$, just by sheafifying it. See \cite[][Proposition 4.5.15.]{riehl2017category} for the construction of colimits in a reflective subcategory.

To explicitly describe $\Xi_0$, we first make a fundamental observation of \dq{local states} of a sheaf. Let $X$ be a $J$-sheaf over a small site $(\C,J)$. For each element $x\in Xc$ for $c \in \C$, there is the unique map $\ceil{x}\colon \ay{c}\to X$ that correspond to $x \in Xc$ by the Yoneda lemma for sheaves \cite[][III.6.(17)]{maclane2012sheaves}.
Throughout this section, we call this morphism $\ceil{x}$ and its epi-part and mono-part of the epi-mono factorization in $\Sh{\C,J}$ are denoted by $q_x, \iota_x$ respectively
\[
\begin{tikzcd}
\ay{c}\ar[rd,\epi,"q_x"']\ar[rr,"\ceil{x}"]&&X&\\
& \biangle{x}\ar[ru,\mono,"\iota_x"']&&\text{in } \Sh{\C,J}.
\end{tikzcd}
\]

The reason we focus on the morphism $\ceil{x}$ and its decomposition $\iota_x \circ q_x$ is the fact that $\iota_x \colon \biangle{x}\subob X$ gives the smallest sub-$J$-sheaf containing $x$. 
%d%下，理論的にもモチベ的にも明記する必要ないかも．
% Although we do not use this fact, it is worth proving to help intuitive understanding.
\begin{lemma}(Minimum subsheaf that contains a chosen element)\label{LemmaGenerateSubsheaf}
For a small site $(\C,J)$, a $J$-sheaf $X$ and  an element $x\in Xc$, a subobject $\iota_x \colon \biangle{x} \subob X$ is the minimum sub-$J$-sheaf of $X$ that contains $x$.
\end{lemma}
\begin{proof}
A sub-$J$-sheaf $\iota\colon S\subob X$ contains $x \in X$ if and only if the morphism $\ceil{x}\colon \ay{c}\to X$ lifts along $\iota\colon S\subob X$
\[
\begin{tikzcd}
&S\ar[d,"\iota",\mono]&\\
\ay{c}\ar[r,"\ceil{x}"]\ar[ur,bend left, \dash]&X&\text{in } \Sh{\C,J}.
\end{tikzcd}
\]
This lemma is immediately implied by this fact and the universal property of the image.
\end{proof}

Recall the explanation of a locally determined cocone given in subsection \ref{subsectionDefinitionIntuitionLSC}, in the context of a Grothendieck topos $\Sh{\C,J}$.
For any locally determined cocone $\{\Lmor_X\colon X \to \Lobj\}$, we can easily prove that $\Lmor_X (x)$ is equal to  $\Lmor_{\langle x\rangle}(x)$ by diagram chasing
\[
\begin{tikzcd}[column sep = tiny]
&\ay{c}\ar[ld,\epi,"q_x"']\ar[rd,"\ceil{x}"]& \\
\biangle{x}\ar[rd,"\Lmor_{\biangle{x}}"']\ar[rr,\mono,"\iota_x"]&&X\ar[ld,"\Lmor_{X}"]& \\
&\Lobj&& \text{in } \Sh{\C,J}.
\end{tikzcd}
\]
% such a smallest subsheaf is expected to play an important role in the description of a local state classifier and is actually used in our construction.
In other words, a locally determined cocone is determined only by the value on the co-subobject of a sheafification of a representable functors $q_x\colon \ay{c}\quo \biangle{x}$. 

These observations lead us to define an approximation $\Xi_0$ of a local state classifier as a collection of all co-subobjects of $\ay{c}$ for each object $c \in \C$. Although a co-subobject is an equivalence class of epimorphisms from $\ay{c}$, by abuse of language, we just write $E$ for a co-subobject (an equivalence class) that $\ay{c}\quo E$ belongs to.
% This object is just an approximation, in fact, it is not $J$-sheaf, in the first place.
%o%r: 気持ち
\begin{definition}[Presheaf $\Xi_0$]
For a small site $(\C,J)$, we define a presheaf $\Xi_0$ over $\C$ as follows:
For each object $c \in \C$, a set $\Xi_0 (c)$ is defined as a set of all co-subobjects of $\ay{c}$. For each morphism $f:c\to c'$ in $\C$, we define a function $\Xi_0(f)\colon \Xi_0(c') \to \Xi_0 (c)$ as a function that sends an element $q:\ay{c'}\quo E$ of $\Xi_0 (c')$ to the epi part of epi-mono factorization (in $\Sh{\C,J}$) of $q \circ \ay{f}$
\[
\begin{tikzcd}
\ay{c}\ar[r,"\ay{f}"]\ar[d,\epi]&\ay{c'}\ar[d,"q",\epi]&\\
\Xi_0 (f)(E)\ar[r,\mono]&E& \text{in } \Sh{\C,J}.
\end{tikzcd}
\]
\end{definition}
Since the number of co-subobjects of $\ay{c}$ is small, it defines a presheaf over $\C$. The functoriality of $\Xi_0 \colon \C^{\op} \to \Set$ is also easily verified by the uniqueness of the epi-mono factorization.

%o%r: 書く? 後のremarkと合わせて，「より単純なものがない話」にするか
% \begin{remark}[$\Xi_0$ may not be a $J$-sheaf]
% One may expect that $\Xi_0$ is a $J$-sheaf and $\Xi_0$ is a local state classifier of $\Sh{\C,J}$. But it is not the case. For example, the topos of directed graph and [...]
% %%r: 書く directed graph
% \end{remark}

To state that $\Xi_0$ is a colimit of the diagram
\[
\begin{tikzcd}
\mo{\Sh{\C,J}}\ar[r,\mono]&\Sh{\C,J}\ar[r,hookrightarrow]&\ps{\C},
\end{tikzcd}
\]
we define a component of a cocone $\{\orb_X\colon X \to \Xi_0\}_{X \in \ob{\Sh{\C,J}}}$. (The symbol $\orb$ is not important and is just taken from the first few letters of \dq{orbit,} inspired by the case of the group action topos.) We use the epi-mono factorization of $\ceil{x}\colon \ay{c}\to X$, as we have mentioned above.
\begin{definition}[Morphism $\orb_X$]
For a $J$-sheaf $X$, we define a morphism $\orb_X \colon X \to \Xi_0$ as follows:
for $x\in \Xi_0 (c)$, $(\orb_X)_c (x)$ is the co-subobject $q_x\colon \ay{c}\quo \biangle{x}$.
\end{definition}

Now we should check that $\orb_X$ defined above is a morphism of presheaves, i.e., the following diagram commutes for an arbitrary $f:c\to c'$ in $\C$:
\[
\begin{tikzcd}
Xc\ar[d,"(\orb_X)_c"]&Xc'\ar[l,"Xf"']\ar[d,"(\orb_X)_{c'}"]&\\
\Xi_0 c & \Xi_0 c'\ar[l,"\Xi_0 f"']& \text{(in } \Set \text{).}
&
\end{tikzcd}
\]
For each $x \in Xc'$, the following commutative diagram verifies it
\[
\begin{tikzcd}[column sep = tiny]
\ay{c}\ar[rr,"\ay{f}"]\ar[d,"q_{x\cdot f}",\epi]\ar[rdd,bend right=80,"\ceil{x\cdot f}"']&&\ay{c'}\ar[d,"q_x",\epi]\ar[ldd,bend left=80,"\ceil{x}"]&\\
\biangle{x\cdot f}\ar[rr,\dash, \mono,"l"]\ar[rd,\mono,"\iota_{x\cdot f}"]&&\biangle{x}\ar[ld,\mono,"\iota_x"]&\\
&X&& \text{in } \Sh{\C,J},
\end{tikzcd}
\]
where $x\cdot f$ denotes $(X(f))(x)$. The existence of the dashed monomorphism $l\colon \biangle{x\cdot f}\subob \biangle{x}$ is implied by the universal property of the image $\iota_{x\cdot f}\colon\biangle{x\cdot f}\subob X$ or equivalently by Lemma \ref{LemmaGenerateSubsheaf}.

%d%f: universality をuniversal propertyへ
By all the above preparations, we are now able to state the central lemma in this subsection.
\begin{lemma}[Universal property of $\Xiz$]\label{LemmaUniversalityOfXiz}
For a small site $(\C,J)$, a family of morphisms $\{\orb_X\colon X \to \Xiz\}_{X \in \ob{\Sh{\C,J}}}$ is a colimit cocone of the following large diagram:
\[
\begin{tikzcd}
\mo{\Sh{\C,J}}\ar[r,\mono]&\Sh{\C,J}\ar[r,hookrightarrow]&\ps{\C}.
\end{tikzcd}
\]
\end{lemma}
% For any locally determined cocone $\{\Lmor_X \colon X \to \Lobj\}$ of $\Sh{\C,J}$, there is the unique morphism $l:\Xi_0 \to \Lobj$ such that 
\begin{proof}

Before proving the universal property, we should prove that the family of morphisms $\{\orb_X\colon X \to \Xiz\}_{X \in \ob{\Sh{\C,J}}}$ is a cocone, in other words, the following diagram
\[
\begin{tikzcd}[column sep = tiny]
X\ar[rr,"m",\mono]\ar[rd,"\orb_X"']&&Y\ar[ld,"\orb_Y"]&\\
&\Xiz&& \text{in } \ps{\C},
\end{tikzcd}
\]
commutes for each monomorphism $m\colon X\to Y$ in $\Sh{\C,J}$. For $c \in \ob{\C}$ and $x\in Xc$, the equation $(\orb_X)_c (x) = (\orb_Y)_c(m(x))$ is followed by the uniqueness of the epi-mono factorization and the following commutative diagram
\[
\begin{tikzcd}[column sep = tiny]
&\ay{c}\ar[ld,"q_x",\epi]\ar[rd,"q_{m(x)}",\epi]\ar[ldd,bend right =80,"\ceil{x}"']\ar[rdd,bend left =80,"\ceil{m(x)}"]&&\\
\biangle{x}\ar[d,"\iota_x", \mono]&&\biangle{m(x)}\ar[d,"\iota_{m(x)}"', \mono]&\\
X\ar[rr,\mono ,"m"]&&Y& \text{in } \Sh{\C,J}.
\end{tikzcd}
\]

Then we prove the universality of this cocone $\{\orb_X\colon X \to \Xiz\}_{X \in \ob{\Sh{\C,J}}}$. Take an arbitrary cocone $\{\Rmor_X\colon X \to \Robj\}_{X \in \ob{\Sh{\C,J}}}$. (Unlike locally determined cocones, $\Robj$ is not necessarily a $J$-sheaf, but is just an object of $\ps{\C}$.) We prove the existence and the uniqueness of a cocone map from $\Xiz$ to $\Robj$.

First, we prove the uniqueness of a cocone map $\Xiz \to \Robj$. Let $l\colon \Xiz \to \Robj$ be a cocone map. For any element $q\colon \ay{c} \quo E$ of $\Xiz c$, we have the following commutative diagram:
\[
\begin{tikzcd}[row sep = tiny]
&&&\Xiz \ar[dd,"l"]&\\
\y{c}\ar[r,"\eta_{\y{c}}"] & \ay{c} \ar[r,"q"] & E \ar[ru,"\orb_X"] \ar[rd,"\Rmor_E"'] & & \\
&&&\Robj&\text{in } \ps{\C},
\end{tikzcd}
\]
since $E$ is a $J$-sheaf. (As we mentioned before, an epimorphism $q$ in $\Sh{\C,J}$ is not necessarily epic in the presheaf topos $\ps{\C}$.) Here, $\eta_{\y{c}}\colon \y{c} \to \ay{c}$ denotes the unit of the adjunction of sheafification. By the definition of $\orb_X$, the Yoneda-corresponding element for the upper half
\[
\begin{tikzcd}[row sep = tiny]
&&&\Xiz &\\
\y{c}\ar[r,"\eta_{\y{c}}"] & \ay{c} \ar[r,"q"] & E \ar[ru,"\orb_X"] & &\text{in } \ps{\C}
\end{tikzcd}
\]
is $q\colon \ay{c} \quo E$, which is an element of $\Xiz c$. Therefore, $l(E) \in \Robj c$ must be the Yoneda-corresponding element for the lower half:
\[
\begin{tikzcd}[row sep = tiny]
\y{c}\ar[r,"\eta_{\y{c}}"] & \ay{c} \ar[r,"q"] & E  \ar[rd,"\Rmor_E"'] & & \\
&&&\Robj&\text{in } \ps{\C}.
\end{tikzcd}
\]
Hence the uniqueness of a cocone map $l\colon \Xiz \to \Robj$ is proved.

Next, we prove the existence of a cocone map. As the proof of the uniqueness suggests, we define $l_{c}\colon \Xiz c \to \Robj c$ as a function that sends $q\colon \ay{c} \quo E$ to the corresponding element of 
\[
\begin{tikzcd}[row sep = tiny]
\y{c}\ar[r,"\eta_{\y{c}}"] & \ay{c} \ar[r,"q"] & E  \ar[r,"\Rmor_E"] & \Robj&\text{in } \ps{\C},
\end{tikzcd}
\]
by the Yoneda lemma. There remain two things to prove. One is that $l$ is a morphism of presheaf $\Xiz \to \Robj$. The other is that $l$ is a morphism of cocone.

We prove that $l$ is a morphism of presheaf $\Xiz \to \Robj$. 
For an arbitrary morphism $f\colon c \to c'$, we should prove the commutativity of the following diagram
\[
\begin{tikzcd}
\Xiz c \ar[d,"l_{c}"]& \Xiz c' \ar[d,"l_{c'}"]\ar[l,"\Xiz f"']&\\
\Robj c&\Robj c'\ar[l,"\Robj f"']& \text{(in } \Set \text{).}
\end{tikzcd}
\]
Take an arbitrary element $q\colon \ay{c'} \quo E$ of $\Xiz c'$. By definition of a presheaf $\Xiz$, we have the next diagram
\[
\begin{tikzcd}
\ay{c}\ar[r,"\ay{f}"]\ar[d,\epi]&\ay{c'}\ar[d,"q",\epi]&\\
\Xi_0 (f)(E)\ar[r,\mono,"i"]&E& \text{in } \Sh{\C,J}.
\end{tikzcd}
\]
Since $i$ is a monomorphism between $J$-sheaves (and the unit natural transformation is natural), the following diagram is commutative
\[
\begin{tikzcd}[column sep = tiny]
\y{c}\ar[rr,"\y{f}"]\ar[d,"\eta_{\y{c}}"]&&\y{c'}\ar[d,"\eta_{\y{c'}}"]&\\
\ay{c}\ar[rr,"\ay{f}"]\ar[d]&&\ay{c'}\ar[d,"q"]&\\
\Xi_0 (f)(E)\ar[rr,\mono,"i"]\ar[rd,"\Rmor_{\Xi_0 (f)(E)}"']&&E\ar[ld,"\Rmor_{E}"]&\\
&\Robj&&\text{in } \ps{\C}.
\end{tikzcd}
\]
By the correspondence of the Yoneda lemma, we have the desired equation.

Lastly, we prove that $l$ is a cocone map. For an arbitrary $J$-sheaf $X$, we prove the commutativity of 
\[
\begin{tikzcd}[column sep =tiny]
&X\ar[ld,"\orb_X"']\ar[rd,"\Rmor_X"]&&\\
\Xiz\ar[rr,"l"]&&\Robj& \text{in } \ps{\C}.
\end{tikzcd}
\]
Take an arbitrary element $x \in Xc$. The corresponding morphism $\y{c} \to X$ for $x \in Xc$ is the composite of the following three morphisms
\[
\begin{tikzcd}
\y{c} \ar[r,"\eta_{\y{c}}"]&\ay{c} \ar[r,"q_x"]&\biangle{x}\ar[r,"\iota_x",\mono]&X &\text{in } \ps{\C}.
\end{tikzcd}
\]
All paths in the following diagram 
\[
\begin{tikzcd}[row sep =tiny]
&&&&\Xiz\ar[dd,"l"]&\\
\y{c}\ar[r,"\eta_{\y{c}}"] & \ay{c} \ar[r,"q_x"] & \biangle{x}\ar[r,"\iota_x",\mono]\ar[rru,bend left,"\orb_{\biangle{x}}"]\ar[rrd,bend right,"\Rmor_{\biangle{x}}"'] & X\ar[ru,"\orb_X"]\ar[rd,"\Rmor_X"'] &&\\
&&&&\Robj&\text{in } \ps{\C}
\end{tikzcd}
\]
from $\y{c}$ to $\Robj$ defines the same morphism, because $\iota_x$ is a monomorphism between $J$-sheaves. The definition of $l$ is also used here. Again by the Yoneda lemma, we have $l_c ((\orb_X)_c (x))=(\Rmor_X)_c (x)$. Thus proof is completed.
\end{proof}

As discussed at the beginning of this section, now we obtain the local state classifier of a Grothendieck topos $\Sh{\C,J}$ by sheafification.
\begin{proposition}[Local state classifier of a Grothendieck topos]
The sheafification $\sha \Xiz$ of $\Xiz$ is a local state classifier of a Grothendieck topos $\Sh{\C,J}$. Each component of the colimit cocone is given by 
\[
\begin{tikzcd}
X\ar[r,"\orb_X"]&\Xiz\ar[r,"\eta_{\Xiz}"]&\sha \Xiz.
\end{tikzcd}
\]
\end{proposition}
\begin{proof}
This is immediately followed by the general construction of colimits in a reflective subcategory. See \cite[][Proposition 4.5.15.]{riehl2017category}.
\end{proof}

We end this subsection by mentioning a few special classes of Grothendieck toposes for which this general construction takes a simpler form.
\begin{example}[Local state classifier of a presheaf topos]\label{LSCofPresheafTopos}
As a particular case, a local state classifier of a presheaf topos is worth mentioning. It is not only because the sheafification functor $\sha$ becomes trivial and the construction becomes simpler but also because we can observe a striking similarity with the construction of the subobject classifier!

For a presheaf topos $\ps{\C}$, $\Xi c$ is a set of all co-subobjects of a representable presheaf $\y{c}$. This is similar to the construction of the subobject classifier $\Omega$ since $\Omega c$ is the set of all subobjects of a representable presheaf $\y{c}$. 

As we mentioned in Introduction (Section \ref{introduction}), our model case for an internal parameterization is that of a subtopos, which takes advantage of the subobject classifier $\Omega$. It is unsurprising that there are theoretical similarities between the roles of $\Omega$ and $\Xi$ in each internal parameterization. However, it is remarkable that there is such an unexpected similarity at the level of construction.

Note that two examples of local state classifiers, Example \ref{ExampleLSCofDirectedGraph} and Example \ref{ExampleOfLSCOfGroupActionTopos} in subsection \ref{subsectionOfExampleOfLSC}, are constructed in this way. (Note that the local state classifier in Example \ref{exampleSpecies} is also constructed in the essentially same way.)
\end{example}

\begin{example}[Local state classifier of a localic Grothendieck topos]\label{LSCofLocalicGrothendieckTopos}
Now we can prove that a local state classifier of a localic Grothendieck topos is the terminal object. (For a simpler proof for a stronger statement, \dq{a Grothendieck topos is localic if and only if its local state classifier is the terminal object}, see subsection \ref{subsectionLSCinLocalic}.)

For a localic Grothendieck topos $\E$, we can take a small site $(P,J)$ where $P$ is a poset and $\E \simeq \Sh{P,J}$ (\cite[][IX.5. Theorem 1 (ii)]{maclane2012sheaves}). Since $\ay{p}$ for $p \in P$ is a subterminal object, $\Xiz p$ is a singleton, i.e., the presheaf $\Xiz$ is a terminal object. Because the sheafification functor $\sha$ is left exact, we conclude the local state classifier $\sha \Xiz$ is terminal.
\end{example}

%d%書く?
\begin{remark}[Sheafification functor may not preserve a local state classifier]\label{RemarkSheafificarionandLSC}
We used sheafification several times in the construction of the local state classifier of a Grothendieck topos. One might wish that the sheafification functor would preserve a local state classifier, and the local state classifier of a Grothendieck topos would be simply given by a sheafification of that of a presheaf topos, but such a construction does not work. A sheafification functor may not preserve a local state classifier, even though it preserves arbitrary colimits and monomorphisms. 

For example, consider the functor $\ceil{E}\colon \1 \to \Par$ from the terminal category $\1$ to the parallel morphisms category $\Par$ (see Example \ref{ExampleLSCofDirectedGraph}), that sends the unique object of $\1$ to $E$. Since this functor is fully faithful, its associated geometric morphism
\[\Set \cong \ps{\1}\to \ps{\Par}\cong \DirGraph,\]
is a geometric embedding (see \cite[][section VII.4.]{maclane2012sheaves}). Its associated sheafification functor $\sha\colon \DirGraph \to \Set$ is just a pre-composite of the functor $\ceil{E}$ and sends a directed graph $X$ to the set of all edges of $X$. Since the local state classifier $\Xi$ of the category of directed graphs has two edges (see Example \ref{ExampleLSCofDirectedGraph}) and the local state classifier of the category of sets is a singleton (see Example \ref{ExampleLSCofSets}), the sheafification functor $\sha$ does not preserve a local state classifier.
\end{remark}

\subsection{Internal semilattice structure}\label{subsectionSemilatticeStructure}
%d%気持ち説明//Internal descriptionになぜ必要か//Presheafやsubobject classifierのそれ//LT-topologyのケースにも触れる．//色々の操作でHQが閉じているためには，それを反映したlocal stateのalgebraic structureを解明する必要がある．
In this subsection, we show that a local state classifier of a cartesian closed category has a natural internal semilattice structure. All definitions, terminology, and facts on internal semilattices used in this subsection are explained in Appendix \ref{AppendixInternalSemilattice}.

In the case of a presheaf topos, the semilattice structure is apparent. Recall that the local state classifier $\Xi$ of a presheaf topos $\ps{\C}$ consists of all co-subobjects of representable presheaves (see Example \ref{LSCofPresheafTopos}). Like the order structure on $\Omega c$ for the subobject classifier $\Omega$ of a presheaf topos, a canonical order structure on $\Xi c$ for each $c \in \ob{\C}$ is given by the usual order between co-subobjects (see Example \ref{semilatticestrinPresheaf} for details). Those semilattice structures on each $\Xi c$ give an internal semilattice structure on $\Xi$. We generalize this semilattice structure to a local state classifier of a cartesian closed category.

But why do we consider internal semilattice structures? There are several possible answers from different perspectives. One answer is that it is technically necessary to realize internal parameterization. Because a hyperconnected quotient is closed under several operations (see Definition \ref{DefOfhyperconn}), its internal parameterization must somehow reflect those operations. In the internal parameterization, the internal semilattice structure of a local state classifier corresponds to finite products (see Lemma \ref{LemmaequiHom}). Recalling that the definition of a local state classifier was inspired only by the operation of taking subobjects (and implicitly coproducts), it is essential to consider an additional structure on it (in this case, semilattice structure) to capture the notion of a hyperconnected quotient.

Another answer is that our model case, the internal parameterization of subtoposes, also takes advantage of the internal semilattice structure on the subobject classifier. One might think that the usual algebraic structure on the subobject classifier $\Omega$ is not just a semilattice but a Heyting algebra structure. It is correct, but when we consider Lawvere-Tierney topologies, only a semilattice structure is used. In fact, a Lawvere-Tierney topology is an idempotent internal semilattice homomorphism on the subobject classifier, not necessarily a Heyting algebra homomorphism. Therefore, from the viewpoint of imitating the model case, it is natural to consider semilattice structure.

% As we see later, consequently, $\Xi$ has an inherent internal semilattice structure \ref{semilatticeonXi}, and desired good subobjects are just internal filters \ref{mainTheorem}. By focusing only on local states, we have compressed the inclusion relation of a topos and hyperconnected quotients into the relation of a local state classifier and its subobjects.

% In this section, we observe that a local state classifier of a cartesian closed category is naturally accompanied by an internal semilattice structure. Since the local state classifier is defined as the colimit, the universal property that makes a map
% \[
% \meet\colon \Xi\times \Xi \to \Xi
% \]
% is seemingly absent. Therefore, we first show that $\Xi^n$ also has the universal property as the colimit.

At first, we realize $\Xi^n$ as a colimit so that we can define meet $\meet\colon \Xi^{2} \to \Xi$ and top $\top\colon \Xi^{0} \to \Xi$.

\begin{lemma}[Universal property of $\Xi^n$]\label{UniXiN}
Let $\C$ be a cartesian closed category with a local state classifier $\{\xi_{X}\colon X\to \Xi\}$.
For any non-negative integer $n\geq 0$, 
\[
\{\xi_{X_1}\times \dots \times \xi_{X_n} \colon X_{1} \times \dots \times X_{n} \to \Xi^n\}
\]
defines the colimit cocone of the functor
\[
\begin{tikzcd}
{\mo{\C}}^{n}\ar[r,\mono]&\C^{n}\ar[r,"\prod_{k=1}^{n}"]&\C.
\end{tikzcd}
\]
\end{lemma}
\begin{proof}
Since $\C$ is cartesian closed, for any object $X\in \ob{\C}$, the functor 
\[
X\times -\colon \C \to \C
\]
is a left adjoint functor. Therefore, $X\times -$ preserves all (not necessarily small) colimits.

By the induction for $n$, we have
\begin{align*}
    \Xi^{n+1} &\cong \Xi^{n} \times \Xi\\
    &\cong (\colim_{\mo{\C}^n} \ X_{1} \times \dots \times X_{n}) \times \Xi\\
    &\cong \colim_{\mo{\C}^n} (X_{1} \times \dots \times X_{n} \times \Xi)\\
    &\cong \colim_{\mo{\C}^n} (X_{1} \times \dots \times X_{n} \times \colim_{\mo{\C}}X)\\
    &\cong \colim_{\mo{\C}^n} (\colim_{\mo{\C}} X_{1} \times \dots \times X_{n} \times X)\\
    % &\cong \colim_{\mo{\C}^n} (\colim_{\mo{\C}} X_{1} \times \dots \times X_{n} \times X)\\
    &\cong \colim_{\mo{\C}^{n+1}} (X_{1} \times \dots \times X_{n} \times X).
\end{align*}
(The base case $n=0$ is a trivial case.) One can easily check that the associated colimit cocone is
\[
\{\xi_{X_1}\times \dots \times \xi_{X_n} \colon X_{1} \times \dots \times X_{n} \to \Xi^n\}.
\]
\end{proof}

Using this universal property, we can define the canonical meet map
\[
\meet_n\colon \Xi^n \to \Xi
\]
as the canonical cocone map from the colimit cocone 
\[
\{\xi_{X_1}\times \dots \times \xi_{X_n} \colon X_{1} \times \dots \times X_{n} \to \Xi^n\}
\]
to the cocone
\[
\{\xi_{(X_1 \times \dots \times X_n)} \colon X_{1} \times \dots \times X_{n} \to \Xi\}.
\]
It is easily verified that the latter is actually a cocone, since $\{\xi_{X}\}$ is locally determined, and a product of monomorphisms is a monomorphism.

In other words, $\meet_n$ is the unique map that makes the following diagram commutative, for any $n$-tuple of objects $(X_1 ,\dots,X_n )$
\[
\begin{tikzcd}[row sep = huge, column sep =tiny]
&X_1 \times \dots \times X_n \ar[ld,"\xi_{X_1}\times \dots \times \xi_{X_n}"']\ar[rd,"\xi_{( X_1 \times \dots \times X_n)}"]&\\
\Xi^n \ar[rr,"\meet_n"]&&\Xi.
\end{tikzcd}
\]
%o%f: 後で編集: \xiの添字の大きさを修正
In particular, $\meet_0$ is the unique map that makes 
\[
\begin{tikzcd}[row sep = huge, column sep =tiny]
&\1 \ar[ld,"\id{\1}"']\ar[rd,"\xi_{\1}"]&\\
\1 \ar[rr,"\meet_0"]&&\Xi
\end{tikzcd}
\]
commutative, which is $\xi_{\1}$. 
% The morphism $\meet_1$ is the unique map that makes 
% \[
% \begin{tikzcd}[row sep = huge, column sep =tiny]
% &X \ar[ld,"\xi_{X}"']\ar[rd,"\xi_{X}"]&\\
% \Xi \ar[rr,"\meet_1"]&&\Xi
% \end{tikzcd}
% \]
% commutative for any $X \in \ob{\C}$, which is $\id{\Xi}$.
% Before checking the 

As we see below, in the proof that the above operations give an internal semilattice structure on $\Xi$, each axiom of semilattice comes from a corresponding natural transformation related to the cartesian structure, as follows:
% the first three axioms of Definition \ref{DefOfBSL} and the last axiom come from the corresponding canonical natural isomorphisms of the symmetric monoidal (in this case, cartesian) structure of $\C$, and the diagonal map, which are monic, respectively.
\begin{align*}
    \top \meet x = x&\leftrightsquigarrow \lambda_{X} \colon \1 \times X \cong X\\
    x \meet \top = x&\leftrightsquigarrow \rho_{X}    \colon X \times \1 \cong X\\
    x\meet (y\meet z)=(x\meet y)\meet z&\leftrightsquigarrow \alpha_{X,Y,Z}\colon X\times (Y\times Z) \cong (X\times Y)\times Z \\
    x\meet y = y \meet x&\leftrightsquigarrow \gamma_{X,Y}\colon X\times Y \cong Y\times X\\
    x = x\meet x&\leftrightsquigarrow \Delta_{X}\colon X \subob X\times X.
\end{align*}
It is theoretically important that all the above natural transformations are monic.

\begin{proposition}[Semilattice structure of $\Xi$]\label{semilatticeonXi}
Let $\C$ be a cartesian closed category with a local state classifier $\{\xi_{X}\colon X\to \Xi\}$. 
Then, the local state classifier $\Xi$ equipped with a top map
\[
\top \coloneqq \meet_0 \colon \1\to \Xi
\]
and a meet map
\[
\meet \coloneqq\meet_2 \colon \Xi\times \Xi \to \Xi
\]
is an internal semilattice.
\end{proposition}
\begin{proof}
We show that the triple $(\Xi,\top,\meet)$ defined above satisfies the four axioms of an internal semilattice in Definition \ref{DefOfBSL}. 

First, we prove the first axiom $\top\meet x = x$, which asserts that 
\[
\begin{tikzcd}
\Xi \ar[rrr,"\id{\Xi}"',bend right]\ar[r,"\lambda_{\Xi}"]&\1 \times \Xi\ar[r,"\top\times \id{\Xi}"]&\Xi \times \Xi\ar[r,"\meet"]&\Xi
\end{tikzcd}
\]commutes, where $\lambda_{A}$ denotes the canonical natural isomorphism $\lambda_{A}\colon \1 \times A \to A$. It is enough to prove that for any $X\in \ob{\C}$, the pre-composite of $\xi_X$
\[
\begin{tikzcd}
X\ar[d,"\xi_X"']&&\\
\Xi \ar[rrr,"\id{\Xi}"',bend right]\ar[r,"\lambda_{\Xi}"]&\1 \times \Xi\ar[r,"\top\times \id{\Xi}"]&\Xi \times \Xi\ar[r,"\meet"]&\Xi
\end{tikzcd}
\] define the same morphism from $X$ to $\Xi$, because $\{\xi_{X}\colon X\to \Xi\}$ is a colimit cocone.
The desired commutativity is implied by the next diagram:
\[
\begin{tikzcd}[row sep=large]
X\ar[d,"\xi_X"']\ar[r,"\lambda_X"]&\1\times X \ar[d,"\id{\1}\times \xi_{X}"']\ar[rd,"\xi_{\1} \times \xi_{X}"] \ar[rrd,bend left,"\xi_{\1 \times X}"]&\\
\Xi \ar[rrr,"\id{\Xi}"',bend right]\ar[r,"\lambda_{\Xi}"]&\1 \times \Xi\ar[r,"\top\times \id{\Xi}"]&\Xi \times \Xi\ar[r,"\meet"]&\Xi.
\end{tikzcd}
\]
In fact, the left square, the middle triangle and the right triangle are commutative, because of the naturality of $\lambda$, the equality $\top = \meet_0 = \xi_{\1}$ and the definition of $\meet = \meet_2$ respectively, and the outer bent trapezoid is also commutative because $\lambda_X$ is monic and $\xi$ is locally determined. The other equation $x \meet \top = x$ in the first axiom is similarly proved, using the canonical natural isomorphism $\rho_{A}\colon \1 \times A \to A$ instead of $\lambda$.

Second, we prove the second axiom $(x\meet y)\meet z=x\meet (y\meet z)$. By a similar argument using Lemma \ref{UniXiN}, it is enough to prove that two possible composed morphisms from $X\times(Y\times Z)$ to $\Xi$ in 
\[
\begin{tikzcd}[row sep =huge, column sep =normal]
X\times (Y \times Z)\ar[rd,"\xi_{X} \times (\xi_{Y} \times \xi_{Z})"]&&(\Xi \times \Xi)\times \Xi\ar[r,"\meet \times \id{\Xi}"]&\Xi \times \Xi\ar[dd,"\meet"]
\\
&\Xi \times (\Xi \times \Xi)\ar[ru,"\alpha_{\Xi,\Xi,\Xi}"]\ar[d,"\id{\Xi}\times \meet"]&&
\\
&\Xi\times \Xi\ar[rr,"\meet"]&&\Xi
\end{tikzcd}
\]
coincide, where $\alpha_{A,B,C}\colon A\times (B\times C)\to (A\times B)\times C$ denotes the canonical isomorphism. By the following diagram
\[
\begin{tikzcd}[row sep =huge, column sep =normal]
&(X\times Y)\times Z\ar[rd,"(\xi_{X} \times \xi_{Y})\times \xi_{Z}"']\ar[rrd,"\xi_{X\times Y}\times \xi_{Z}"]\ar[rrddd,controls={+(3,1.5) and +(4,3)},"\xi_{(X\times Y) \times Z}"]&&
\\
X\times (Y \times Z)\ar[rd,"\xi_{X} \times (\xi_{Y} \times \xi_{Z})"]\ar[ru,"\alpha_{X,Y,Z}"]\ar[rdd,"\xi_{X}\times \xi_{Y\times Z}"']\ar[rrrdd,controls={+(-1.5,-4) and +(-3,-3)},"\xi_{X\times (Y \times Z)}"']&&(\Xi \times \Xi)\times \Xi\ar[r,"\meet \times \id{\Xi}"']&\Xi \times \Xi\ar[dd,"\meet"]
\\
&\Xi \times (\Xi \times \Xi)\ar[ru,"\alpha_{\Xi,\Xi,\Xi}"]\ar[d,"\id{\Xi}\times \meet"]&&
\\
&\Xi\times \Xi\ar[rr,"\meet"]&&\Xi,
\end{tikzcd}
\]
the second axiom is proved, because of the naturality of $\alpha_{X,Y,Z}$, the definition of $\meet$, and the fact that $\xi$ is locally determined and $\alpha$ is monic.%d%このdiagramズレる問題，やっぱりなんとかしたいね．

To prove the third axiom $x\meet y = y \meet x$, it is enough to prove that 
\[
\begin{tikzcd}[row sep = tiny]
X\times Y\ar[r,"\xi_{X} \times \xi_{Y}"]&\Xi \times \Xi\ar[rd,"\meet"]\ar[dd,"\gamma_{\Xi,\Xi}"]&\\
&&\Xi\\
&\Xi \times \Xi\ar[ru,"\meet"']&\
\end{tikzcd}
\]
defines the same morphism, where $\gamma_{A,B}\colon A\times B \cong B\times A$ denotes the canonical isomorphism given by the cartesian symmetric monoidal structure of $\C$ (namely, $\gamma_{A,B} = \gen{\mathrm{pr}_2,\mathrm{pr}_1}$). Checking the commutativity of the following diagram
\[
\begin{tikzcd}[row sep = tiny]
X\times Y\ar[r,"\xi_{X} \times \xi_{Y}"]\ar[dd,"\gamma_{X,Y}"]\ar[rrd,bend left=50,"\xi_{X\times Y}"]&\Xi \times \Xi,\ar[rd,"\meet"]\ar[dd,"\gamma_{\Xi,\Xi}"]&\\
&&\Xi\\
Y\times X\ar[r,"\xi_{Y} \times \xi_{X}"]\ar[rru,bend right=50,"\xi_{Y\times X}"']&\Xi \times \Xi\ar[ru,"\meet"']&\
\end{tikzcd}
\]
using the definition of $\meet$ and the fact that $\gamma$ is monic, the third axiom is proved.

For the last axiom $x=x\meet x$, it is enough to prove
\[
\begin{tikzcd}
X\ar[d,"\xi_{X}"]&\\
\Xi\ar[r,\mono,"\Delta_{\Xi}"]\ar[rd,"\id{\Xi}"']&\Xi \times \Xi\ar[d,"\meet"]\\
&\Xi
\end{tikzcd}
\]
commutes for each $X \in \ob{\C}$, where $\Delta_{A}\colon A \subob A\times A$ denotes the diagonal map. Similarly, the commutativity of the following diagram
\[
\begin{tikzcd}
X\ar[d,"\xi_{X}"']\ar[r,\mono,"\Delta_{X}"]&X\times X\ar[d,"\xi_{X} \times \xi_{X}"']\ar[dd,bend left=50, "\xi_{X\times X}"]\\
\Xi\ar[r,\mono,"\Delta_{\Xi}"]\ar[rd,"\id{\Xi}"']&\Xi \times \Xi\ar[d,"\meet"']\\
&\Xi
\end{tikzcd}
\]
and the fact that $\Delta_{X}$ is monic, imply the fourth axiom. Now we completed the proof.
\end{proof}

As we promised at the beginning of this subsection, we observe the semilattice structure on a local state classifier of a presheaf topos.
\begin{example}[Presheaf topos]\label{semilatticestrinPresheaf}
% The semilattice structure of $\Xi$ for a presheaf topos can be explicitly described.
% Let $\C$ be a small category. 
The local state classifier $\Xi$ of the presheaf topos $\ps{\C}$ is explicitly described in Example \ref{LSCofPresheafTopos}. According to it, $\Xi c$ for $c \in \ob{\C}$ is the set of all co-subobjects of representable presheaves $\y{c}$. This set has the natural partial order structure given by the morphisms between co-subobjects. More precisely, for two given co-subobjects $q_{i}\colon \y{c}\quo E_i$ for $i=0,1$, an inequality $E_0 \leq E_1$ holds if and only if there is a morphism $f\colon E_0 \to E_1$ such that 
\[
\begin{tikzcd}[row sep = tiny]
&E_0\ar[dd,"f"]\\
\y{c}\ar[ru, \epi,"q_0"]\ar[rd,\epi ,"q_1"]&\\
&E_1
\end{tikzcd}
\]
commutes.

This poset $\Xi c$ has all finite meets, then is a semilattice. In fact, the meet of co-subobjects $q_{i}\colon \y{c}\quo E_i$ for $i=1, \dots n$ is given by the epi part of the epi-mono factorization of \[\langle q_1, \dots, q_n \rangle \colon  \y{c} \to E_1 \times \dots \times E_n.\]

By straightforward calculation, it is easily verified that this semilattice structure defines the internal semilattice structure on $\Xi$ and coincides with the one given in the Proposition \ref{semilatticeonXi}.
\end{example}

\begin{remark}[Local state classifier may not be a Heyting algebra]
The analogy with the subobject classifier might lead one to imagine that $\Xi$ has an internal Heyting algebra structure, and the above internal semilattice structure is the restriction of it. But, it is not the case.

For example, the group action topos $\ps{G}$ for some group $G$ gives a counterexample. In Example \ref{ExampleOfLSCOfGroupActionTopos}, we have seen that a local state classifier of $\ps{G}$ is given by the set of all subgroups of $G$. As a special case of Example \ref{semilatticestrinPresheaf}, its canonical internal semilattice structure coincides with the one given by the usual inclusion relation of subgroups.
The semilattice structure of $\SubGrp{G}$ is not necessarily able to be extended to a Heyting algebra structure. For example, $\SubGrp{\Z / 2\Z \times  \Z / 2\Z}$ has the following Hasse diagram:
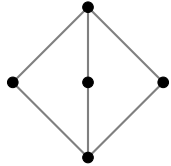
\[
\begin{tikzpicture}
\draw[gray, thick] (0,1) -- (1,0);
\draw[gray, thick] (0,1) -- (0,0);
\draw[gray, thick] (0,1) -- (-1,0);
\draw[gray, thick] (0,-1) -- (1,0);
\draw[gray, thick] (0,-1) -- (0,0);
\draw[gray, thick] (0,-1) -- (-1,0);
\filldraw[black] (0,0) circle (2pt);
\filldraw[black] (-1,0) circle (2pt);
\filldraw[black] (1,0) circle (2pt);
\filldraw[black] (0,1) circle (2pt);
\filldraw[black] (0,-1) circle (2pt);
\end{tikzpicture}
\]
and is not a Heyting algebra, since it is not distributive.
\end{remark}

\section{Internal parameterization of hyperconnected quotients}\label{SecIPH}
The aim of this section is to prove the following main theorem, internal parameterization of hyperconnected quotients. (See Appendix \ref{AppendixInternalSemilattice} for the definition and properties of internal filters.)
\begin{theorem}[Main theorem]\label{mainTheorem}
Let $\E$ be a topos with a local state classifier $\{\xi_{X}\colon X\to\Xi\}_{X\in \ob{\E}}$ (for example, an arbitrary Grothendieck topos). Then the following three concepts correspond bijectively.
\begin{enumerate}
    \item hyperconnected quotients of $\E$
    \item internal filters of $\Xi$
    \item internal semilattice homomorphisms $\Xi\to \Omega$
\end{enumerate}
\end{theorem}
% Here, we make remarks on several points. First, 
% There are several reasons why we explicitly mention internal filters, although bijective correspondence between (2) internal filters and (3) internal homomorphisms is immediately implied by a general proposition (Proposition \ref{PropUniversalityOfSubobjectClassifierAsSemilattice}). 
%o%r: うまく書けるなら，main theoremについて注釈を入れたい．Internal homはsubtoposとの類似で必須．internal filterは，3章で述べた"good subobject"なんだということ．
From the point of view of the analogy to the case of subtoposes, the correspondence between (1) and (3) is essential: in the case of subtopos, (1) is a subtopos and (3) is a Lawvere-Tierney topology. In order to claim that this is the internal parameterization of hyperconnected quotients, the correspondence between (1) and (3) is enough. Then, why do we explicitly write (2) internal filters? One reason is that an internal filter is nothing but a \dq{good subobject} promised in the end of subsection \ref{subsectionNecessityandInevita}. Another reason is that in some concrete examples, internal filters are easier to deal with (see section \ref{SecExampleApplication}).

Our proof is divided into two steps, subsection \ref{subsectionBroaderCorrespondence} and subsection \ref{subsectionMainTheorem}. Eventually, we want to prove that (3) internal semilattice homomorphisms $\Xi \to \Omega$ correspond to (1) hyperconnected quotients. (The correspondence between (2) internal filters and (3) internal homomorphisms is immediately obtained by a general fact, Proposition \ref{PropUniversalityOfSubobjectClassifierAsSemilattice}) 
% But we prove it in the second step (subsection \ref{subsectionMainTheorem}).
However, before proving it, as the first step, we prove a broader correspondence between morphisms $\Xi \to \Omega$ (not necessarily preserve semilattice structure) and what we call \dq{coherent families} in subsection \ref{subsectionBroaderCorrespondence}. Our main theorem is obtained by restricting this broader correspondence. In other words, the first step is the construction of the correspondence, and the second step is the restriction of it using internal semilattice structures of $\Xi$ and $\Omega$.

% In subsection \ref{subsectionBroaderCorrespondence}, we explain the first step: broader correspondence. In this step, we prove a broader correspondence between mor
\subsection{Broader correspondence}\label{subsectionBroaderCorrespondence}
% Before we prove the main theorem, which states that the following three concepts correspond bijectively
% \begin{itemize}
%     \item hyperconnected quotients
%     \item internal filters of $\Xi$
%     \item internal semilattice homomorphism $\Xi\to \Omega$,
% \end{itemize}
In this subsection, we construct the broader correspondence as a preparation for our main theorem.
% between the following three concepts
% \begin{itemize}
%     \item \emph{coherent} families of subobjects, which are defined below
%     \item subobjects of $\Xi$
%     \item morphisms $\Xi\to \Omega$
% \end{itemize}. 
To state the correspondence rigorously, we introduce one terminology, a coherent family of subobjects.
%o%r: morphism \Xi \to \Omega の対応物を考えて導入しても良かったよね．Coherent familyの話ね．
\begin{definition}[Coherent family]
For a topos $\E$, a family of subobjects of all objects in $\E$
\[
\{m_X \colon  S_X \subob X\}_{X\in \ob{\E}}
\]
 is said to be \emph{coherent}, if for any monomorphism $l\colon  X\to Y$ in $\E$, $S_X$ is a pullback of $S_Y$ along $l$
 \[
 \begin{tikzcd}
 S_X\ar[r,dotted,"\exists"]\ar[d,\mono,"m_X"]&S_Y\ar[d,\mono,"m_Y"]\\
 X\ar[r,\mono,"l"]&{Y.}
 \end{tikzcd}
 \]
\end{definition}

Before stating and proving the broader correspondence, we clarify that the notion of coherent families generalizes hyperconnected quotients.

\begin{lemma}[Hyperconnected quotient defines a coherent family]\label{LemmaCoherentFamilyandHyperconnectedQuotient}
Let $\E$ be a topos and $\E\to \F$ be a hyperconnected geometric morphism.
Let $G$ denote the corresponding left exact comonad, and $\epsilon\colon G\Rightarrow \id{\E}$ denote its counit.

Then, $\{\epsilon_X \colon  GX\to X\}_{X\in \ob{\E}}$ is a coherent family of subobjects. Moreover, the hyperconnected quotient is able to be recovered (as a replete full subcategory)
% (up to the suitable notion of equivalence, in detail, see \ref{EquivOfHyperquotients})
from its induced coherent family $\{\epsilon_X \colon  GX\to X\}$ by collecting objects such that $\epsilon_X$ is isomorphic.
\end{lemma}
\begin{proof}
At first, by Definition \ref{DefOfhyperconn}, $\{\epsilon_X \colon  GX\to X\}$ is a family of monomorphisms, i.e., a family of subobjects. We prove that for any monomorphism $\iota\colon  X\subob Y$, 
\[
\begin{tikzcd}
GX\ar[r,\mono, "G\iota"]\ar[d,\mono, "\epsilon_X"]&GY\ar[d,\mono,"\epsilon_Y"]\\
X\ar[r,\mono,"\iota"]&Y
\end{tikzcd}
\]
is a pullback diagram. Let 
\[
\begin{tikzcd}
P\ar[r,\mono]\ar[d,\mono]&GY\ar[d,\mono,"\epsilon_Y"]\\
X\ar[r,\mono,"\iota"]&Y
\end{tikzcd}
\]
be a pullback diagram. Then, $GX \leq P$ holds in the poset of subobjects of $X$, because of the universal property of pullback $P$. The converse inequality $P\leq GX$ is implied by the universal property of the counit $\epsilon_X$ because $P$ is a subobject of $GY$ and the essential image of the inverse image functor of a hyperconnected geometric morphism is closed under taking subobjects (see Definition \ref{DefOfhyperconn}).

The latter statement holds for general coreflective subcategories, not limited to hyperconnected quotients.
\end{proof}

By the above lemma, we can now see that the following correspondence is a broader version of the main theorem.

\begin{proposition}[Broader correspondence]\label{BroaderCorrespondence}
Let $\E$ be a topos with a local state classifier $\{\xi_{X}\colon X\to\Xi\}_{X\in \ob{\E}}$. Then the following three concepts correspond bijectively.
\begin{enumerate}
    \item coherent families of subobjects
    \item subobjects of $\Xi$
    \item morphisms $\Xi\to \Omega$
\end{enumerate}
\end{proposition}
\begin{proof}
The correspondence between 2 and 3 is obvious from the universal property of the subobject classifier $\Omega$. For a given family of subobjects $\{m_{X}\colon S_{X}\subob X\}$, it is coherent if and only if the corresponding characteristic morphisms $\{\chi_{S_X} \colon X\to \Omega\}$ is locally determined. Then, the correspondence between 1 and 3 is verified.
\end{proof}

\subsection{Main theorem}\label{subsectionMainTheorem}
% For restricting the broader correspondence given in proposition \ref{BroaderCorrespondence} to the intended one, we translate the semilattice homomorphism into the 
In this section, we prepare two lemmas and prove our main theorem (Theorem \ref{mainTheorem}).

The first lemma states that a morphism $f\colon \Xi \to \Omega$ is an internal semilattice homomorphism if and only if the corresponding coherent family of subobjects is compatible with finite products.
\begin{lemma}\label{LemmaequiHom}
Let $\E$ be a topos with a local state classifier $\{\xi_{X}\colon X\to\Xi\}_{X\in \ob{\E}}$ and $\{m_X \colon  S_X \subob X\}$ be a coherent family of subobjects. Then for the corresponding morphism $f\colon \Xi \to \Omega$ given by Proposition \ref{BroaderCorrespondence}, 
\begin{enumerate}
    \item f preserves $\top$ if and only if $m_{\1}\colon S_{\1}\subob X$ is isomorphic.
    \item f preserves $\meet$ if and only if for any $X,Y \in \ob{\E}$, $m_{X\times Y}\colon S_{X\times Y}\subob X\times Y$ and ${m_X}\times {m_Y}\colon  {S_X}\times{S_Y}\subob X\times Y$ are equal as subobjects of $X\times Y$.
\end{enumerate}
\end{lemma}

\begin{proof}
\begin{enumerate}
    \item $f$ preserves $\top$ means that the following diagram commutes.\[
    \begin{tikzcd}[row sep = tiny]
    &\Xi\ar[dd,"f"]\\
    \1\ar[ru,"\xi_{\1}"]\ar[rd,"\true"']&\\
    &\Omega
    \end{tikzcd}\]
    By taking corresponding subobjects of the above two morphisms $\1 \to \Omega$, the former part of the lemma is proved.
    
    \item $f$ preserves $\meet$ means that 
    \[
    \begin{tikzcd}
    \Xi \times \Xi\ar[r,"\meet"]\ar[d,"f\times f"']&\Xi\ar[d,"f"]\\
    \Omega \times \Omega\ar[r,"\meet"]& \Omega
    \end{tikzcd}
    \]
    commutes. By Lemma \ref{UniXiN}, this is equivalent to the commutativities of 
    \[
    \begin{tikzcd}
    X\times Y\ar[d,"\xi_{X} \times \xi_{Y}"']&\\
    \Xi \times \Xi\ar[r,"\meet"]\ar[d,"f\times f"']&\Xi\ar[d,"f"]\\
    \Omega \times \Omega\ar[r,"\meet"]& \Omega
    \end{tikzcd}
    \] for all $X,Y \in \ob{\E}$.
    By definition of $\meet$, it is equivalent to the commutativity of 
    \[
    \begin{tikzcd}
    X\times Y\ar[d,"\xi_{X} \times \xi_{Y}"']\ar[rd, "\xi_{X\times Y}", bend left]&\\
    \Xi \times \Xi\ar[d,"f\times f"']&\Xi\ar[d,"f"]\\
    \Omega \times \Omega\ar[r,"\meet"]& \Omega .
    \end{tikzcd}
    \] By taking corresponding subobjects and elementary calculus of pullbacks, one can check that it is equivalent to \[S_X \times S_Y = S_{X\times Y}\] as subobjects of $X\times Y$. 
    % $(S_{X} \times Y) \meet (X \times S_Y ) = S_X \times S_Y$ is followed by elementary calculus of pullbacks, then 
    The latter part of the lemma is also proved.
\end{enumerate}
\end{proof}

\begin{remark}%d%r: このRemarkで何が言いたいのか一言で明確にしたい
The conditions (1) and (2) in Lemma \ref{LemmaequiHom} is properly stronger than the closedness by finite products. More precisely, although the lemma implies that if $f\colon \Xi \to \Omega$ is an internal semilattice homomorphism, the full subcategory that consists of $\{S_{X}\mid X \in \ob{\E}\}$ is closed under finite products, the reverse implication is not true. 

For example, let $\E$ be the topos of $\Z$-action $\ps{\Z}$ and $S_X \subob X$ be the subobject consisting of orbits of $X$ whose cardinalities are $1$ or infinite. Then, $\{S_X \subob X\}$ is coherent, and $\{S_{X}\mid X \in \ob{\E}\}$ is closed under finite product, but the corresponding morphism $\Xi \to \Omega$ is not an internal semilattice homomorphism. In fact, let $X=\Z$ and $Y=\Z / 2\Z$ with the usual action by $\Z$, then $S_{X\times Y}$ is $X\times Y$ itself, but $S_X \times S_Y =\Z \times \emptyset =\emptyset$.
\end{remark}

Now we turn to the second lemma.
Recall that 
\[
\begin{tikzcd}[row sep =small]
X\ar[rr,"\xi_{X}",""'{name=A}]\ar[rd,"f"',\mono]&&\Xi\\
&Y\ar[ru,"\xi_{Y}"']&
\end{tikzcd}
\]is commutative for any monomorphism $f\colon X\subob Y$, by definition of $\Xi$.
The second lemma can be regarded as an analogous property for every morphism $f\colon X\to Y$, which is not necessarily monic. To state the lemma rigorously, notice that for any object $X \in \ob{\E}$, the hom-set $\E(X, \Xi)$ has a semilattice structure induced by that of $\Xi$, as long as $\E$ has the local state classifier $\Xi$. In particular, we have an order structure on $\E(X, \Xi)$. Explicitly, for a parallel pair of morphisms $f,g\colon  X\to \Xi$, the inequality $f\leq g$ is defined to be $f=f\meet g$. In other words, the inequality $f\leq g$ means that the following diagram 
% \[
% \begin{tikzcd}
% X\ar[r,"{\langle f,g \rangle}"]\ar[rd,"f"'] &\Xi \times \Xi \ar[d,"\meet "]\\
% &\Xi
% \end{tikzcd}
% \]
\[
\begin{tikzcd}[row sep = tiny]
X\ar[rd,"f"]\ar[dd,"{\langle f,g\rangle}"']&\\
&\Xi\\
\Xi\times \Xi\ar[ru,"\meet"']
\end{tikzcd}
\]
commutes.

%d%f: 上のdiagram書き換え

\begin{lemma}\label{orderLemma}
Let $\E$ be a topos with a local state classifier $\{\xi_{X}\colon X\to\Xi\}_{X\in \ob{\E}}$. Then, for any morphism $f\colon X\to Y$ in $\E$, an inequality $\xi_{X} \leq \xi_Y \circ f$ holds.
\[
\begin{tikzcd}[row sep =small]
X\ar[rr,"\xi_{X}",""'{name=A}]\ar[rd,"f"']&&\Xi\\
&Y\ar[ru,"\xi_{Y}"']\ar[to = A ,phantom, "{\rotatebox{90}{$\geq$}}"]&
\end{tikzcd}
\]
\end{lemma}
\begin{proof}
By definition, $\xi_{X} \leq \xi_Y \circ f$ is equivalent to the commutativity of 
\[
\begin{tikzcd}[row sep = tiny]
X\ar[rd,"\xi_{X}"]\ar[dd,"{\langle \xi_{X},\xi_{Y} \circ f\rangle}"']&\\
&\Xi\\
\Xi\times \Xi\ar[ru,"\meet"']&.
\end{tikzcd}
\]
Then, it is enough to prove the commutativities of the following two triangles.
\[
\begin{tikzcd}[column sep = large]
X\ar[rd,"\xi_{X}"]\ar[d,"{\langle \id{X},f\rangle}"',\mono]&\\
X\times Y \ar[r,"\xi_{X \times Y}"]\ar[d,"\xi_{X} \times \xi_{Y}"']& \Xi\\
\Xi \times \Xi  \ar[ru,"\meet"']&
\end{tikzcd}
\] The upper one is commutative, since ${\langle \id{X},f\rangle}$ is monic and the cocone $\{\xi_{X}\colon X\to\Xi\}_{X\in \ob{\E}}$ is locally determined. The lower one is commutative by the definition of $\meet$.
\end{proof}
(The same proof works for a cartesian closed category $\E$ not limited to toposes, but we do not need that generality here.)

With these two lemmas, we proceed to the proof of the main theorem.
%d%f: willを減らす
\begin{proof}[proof of Theorem \ref{mainTheorem}]
We restrict the correspondence given by Proposition \ref{BroaderCorrespondence} to the desired one. The correspondence between (2) internal filters and (3) internal semilattice homomorphisms is just a special case of Proposition \ref{PropUniversalityOfSubobjectClassifierAsSemilattice}. Therefore, it is enough to prove that a coherent family of subobjects gives a hyperconnected quotient (by Lemma \ref{LemmaCoherentFamilyandHyperconnectedQuotient}) if and only if the corresponding morphism $f:\Xi \to \Omega$ is an internal semilattice homomorphism (or equivalently, the corresponding subobject of $\Xi$ is an internal filter).
% We prove the correspondence between 1 and 2$=$3.

First, we prove that for any hyperconnected quotient $Q$ of $\E$, the corresponding morphism  $f\colon \Xi \to \Omega$ is an internal semilattice homomorphism.
% First, we prove that for any hyperconnected quotient $Q$ of $\E$ that is induced by the idempotent left exact comonad $G\colon \E \to \E$, the morphism $f\colon \Xi \to \Omega$ corresponding to the counit $\{\epsilon_{X} \colon  GX \subob X\}_{X\in \ob{\E}}$ is an internal semilattice homomorphism. 
Let $G\colon \E \to \E$ denote the left exact comonad induced by hyperconnected quotient $Q$ and $\epsilon\colon G \to \id{\E}$ denote its counit.
By the first lemma (Lemma \ref{LemmaequiHom}), it is enough to prove that $\epsilon_{\1}\colon G\1 \subob \1$ is isomorphic, and for any objects $X,Y \in \ob{\E}$ two subobjects $\epsilon_{X\times Y}\colon G(X\times Y) \subob X\times Y$ and $\epsilon_{X} \times \epsilon_{Y} \colon GX \times GY \subob X\times Y$ define the same subobject of $X \times Y$. Both conditions are implied by the fact that $G$ preserves finite products.

Second, we prove that for each internal filter $k\colon F \subob \Xi$, the corresponding coherent family of subobjects $\{m_{X}\colon S_X \subob X\}$ defines a hyperconnected quotient of $\E$. In detail, we prove the following two things.
\begin{itemize}
    \item The full subcategory $Q$ spanned by the objects $X$ whose $\xi_{X}\colon  X \to \Xi$ lifts along $k\colon F\subob \Xi$
    \[
    \begin{tikzcd}
    &F\ar[d,\mono, "k"]\\
    X\ar[ru,\dash]\ar[r,"\xi_{X}"]&\Xi
    \end{tikzcd}
    \]
    defines hyperconnected quotient of $\E$.
    \item Furthermore, the corresponding coherent family $\{m_{X}\colon S_X \subob X\}$ gives the counit of the hyperconnected geometric morphism.
\end{itemize}
Since $m_{X}\colon S_{X}\subob X$ is the pullback of $k$ along $\xi_{X}$
\[
\begin{tikzcd}
    S_{X}\ar[d,"m_{X}",\mono]\ar[r]\arrow[dr, phantom,""{pullback}, very near start]&F\ar[d,\mono, "k"]\\
    X\ar[r,"\xi_{X}"]&\Xi,
\end{tikzcd}
\] an object $X$ of $\E$ belongs $\Q$ if and only if $m_{X}$ is isomorphic. Therefore, it is followed that $\Q$ is closed under finite products, by the first lemma (Lemma \ref{LemmaequiHom}).

$\Q$ is also closed under taking subobjects because if $Y$ is an object in $\Q$ and $s\colon X\subob Y$ is a subobject, then $\xi_{X}$ lifts along $k$ as the following diagram
\[
\begin{tikzcd}
&&F\ar[d,"k",\mono]\\
X\ar[r,"s",\mono]\ar[rr,bend right, "\xi_{X}"'] &Y\ar[r,"\xi_{Y}"]\ar[ru, \dash]&\Xi.
\end{tikzcd}
\]

Since an equalizer is a special type of subobject, now we have proven that $\Q$ is closed under finite limits and subobjects.

Next, we prove that the embedding functor $\Q \to \E$ has a right adjoint and $\{m_{X} \colon  S_{X}\subob X\}_{X\in \ob{\E}}$ gives its counit. It is equivalent to say that for any object $X$ in $\Q$, $Y$ in $\E$, and a morphism $f\colon X\to Y$, $f$ has a (necessarily unique) lift along $m_{Y}$
\[
\begin{tikzcd}
&S_{Y}\ar[d,\mono,"m_{Y}"]\\
X\ar[r,"f"]\ar[ru,\dash]&Y,
\end{tikzcd}
\]
because one can easily check that $S_Y$ belongs to $\Q$.

Since the object $X$ belongs to $\Q$, the morphism $\xi_{X}$ lifts along $k\colon F\subob\Xi$. The fact that he internal filter $F$ is upward closed (see the third condition of Definition \ref{InternalFilter}) and the second lemma (Lemma \ref{orderLemma}) imply that the composite morphism $\xi_{Y}\circ f$ lifts along $k$
\[
\begin{tikzcd}
&&F\ar[d,"k",\mono]\\
X\ar[r,"f"]\ar[rru,\dash,bend left]&Y\ar[r,"\xi_{Y}"]&\Xi.
\end{tikzcd}
\]

Then, by the universal property of $S_X$ as a pullback, $f$ lifts along $m_{Y}$
\[
\begin{tikzcd}
&S_{Y}\ar[r]\ar[d,\mono,"m_{Y}"]\arrow[dr, phantom,""{pullback}, very near start]&F\ar[d,"k",\mono]\\
X\ar[r,"f"]\ar[ru,\dash]&Y\ar[r,"\xi_{Y}"]&\Xi.
\end{tikzcd}
\]
Therefore, $f$ lifts along $m_{Y}$, and $m_{Y}$ defines the component of the counit.

So far, we have proven that $\Q$ is a full subcategory closed under taking finite limits and subobjects, and the embedding functor $\Q \to \E$ is left exact left adjoint functor. Since $\Q$ is comonadic over $\E$ by a left exact comonad, $\Q$ is a topos (\cite[see][V.8.]{maclane2012sheaves}). Consequently, $\Q$ is a hyperconnected quotient of $\E$, since it is closed under subobjects. The proof of the main theorem is completed.
\end{proof}

%o%m: relation to Quotient Systems

%o%m: induced by bof functor citing Rosenthal
\section{Examples and Applications}\label{SecExampleApplication}
In this subsection, we list some examples and applications.

\subsection{Number of hyperconnected quotients}
%d%number of hyperconnected quotients of locally small topos with LSC.
First, we prove the following immediate corollary about the number of hyperconnected quotients. Although the case for a Grothendieck topos can be implied by \cite{rosenthal1982quotient}, our corollary is applicable to a broader class of toposes and directly implied by the internal parameterization, just like the case of subtoposes.

\begin{corollary}[Smallness of the number of hyperconnected quotients]\label{CorNumberofHyperQuotients}
For a locally small topos with a local state classifier $\E$ (for example, an arbitrary Grothendieck topos), the  number of hyperconnected quotients of $\E$ is small.
\end{corollary}
\begin{proof}
Since hyperconnected quotients correspond bijectively to the internal semilattice homomorphisms from the local state classifier $\Xi$ to the subobject classifier $\Omega$ (by our main theorem, Theorem \ref{mainTheorem}), the number of them is not larger than the cardinality of the hom-set $\E(\Xi,\Omega)$, which is small.
\end{proof}

\subsection{Toy example: the topos of directed graphs}
As a toy example, we classify all hyperconnected quotients of the topos of directed graphs $\DirGraph = \ps{\Par}$, using the main theorem. Before regarding the local state classifier of it, it is not hard to find two hyperconnected quotients, induced by full and bijective on objects functors (see Example \ref{HQfromfbo}). There are two full and bijective on objects functor from the parallel morphisms category $\Par$. The obvious one is $\id{\Par}$ and the other is $\Par\to \2$, where $\2$ denotes the two elements totally ordered set. Their corresponding hyperconnected quotients are $\DirGraph$ itself and $\ps{\2}$ respectively.

Now, we turn to a local state classifier. As observed in Example \ref{ExampleLSCofDirectedGraph}, the local state classifier $\Xi$ is the following directed graph:
\[
\begin{tikzcd}
\bullet\ar[loop left,"\bL"]\ar[loop right,"\bN ."]
\end{tikzcd}
\]
Since the internal semilattice structure is given by \[\bL \geq \bN ,\] there are exactly two internal filters, which are the maximum filter $\id{\Xi}\colon\Xi\to\Xi$ itself and the minimum filter $\top\colon\1\to \Xi$, visualized as follows
\[
\begin{tikzcd}
\bullet\ar[loop left,"\bL"]\ar[loop right,"\bN "]\\
\bullet\ar[loop left,"\bL"].
\end{tikzcd}
\]
The corresponding hyperconnected quotients of them are the two mentioned in the above paragraph. We can conclude that there are no other hyperconnected quotients, by the main theorem.

\subsection{Local state classifier of a localic Grothendieck topos}\label{subsectionLSCinLocalic}
Let $\E$ be a topos with a local state classifier $\Xi$. Then, the identity $\id{\Xi}\colon \Xi\to \Xi$ gives the maximum internal filter, and the top  global section $\top \colon \1\to \Xi$ gives the minimum internal filter. The corresponding hyperconnected quotients are the maximum hyperconnected quotient, i.e., $\E$ itself, and the minimum hyperconnected quotient.

By the main theorem and Example \ref{HQoflocalictopoi}, we got the following corollary, which is the converse statement of Example \ref{LSCofLocalicGrothendieckTopos}.
\begin{corollary}[] %\label{CorLocalicIFFterminal}
For a Grothendieck topos $\E$, $\E$ is localic if and only if its local state classifier is the terminal object.
\end{corollary}
See also Example \ref{ExampleTerminalLSC} for categories whose local state classifier is terminal.
%o%m: the minimum hyperconnected quotientは、localic reflection上の層の圏になるはず?書いて

\subsection{Quotients of Boolean toposes}\label{CorBoolean}
As we mentioned in Introduction (section \ref{introduction}), Lawvere's original question in \cite{OpenLawvere} sought the internal parameterization of all quotients, not hyperconnected quotients.

The difference between the two classes of quotients disappears for some toposes, including all Boolean toposes. 

\begin{proposition}\label{PropQuotientofBooleanTopos}
Let $\E$ be a Boolean topos. Every quotient of $\E$ is hyperconnected.
\end{proposition}
\begin{proof}
Let $\Q$ be its quotient, regarded as a full subcategory of $\E$. The subobject classifier of a Boolean topos $\E$ is the coproduct of two copies of terminal objects $1\coprod 1$ (See \cite[][VI.1.]{maclane2012sheaves}). Since $\Q$ is closed under finite limits and finite colimits of $\E$, the subobject classifier of $\E$, $1\coprod 1$ belongs to $\Q$. By the universal property of the subobject classifier (in $\E$) and the fact that $\Q$ is closed under pullback, $\Q$ is closed under taking subobjects, i.e., hyperconnected.
%o%m: quotientのlim, colimに触れておくか: これ何？
\end{proof}

Then, we had the internal parameterization of all quotients for a Boolean topos with a local state classifier, just by erasing the adjective \dq{hyperconnected.}

\begin{corollary}[Internal parameterization of quotients of a Boolean topos]\label{CorInternalparameterizationOfQuotientsOfBooleanTopos}
Let $\E$ be a Boolean topos with a local state classifier $\{\xi_{X}\colon X\to\Xi\}_{X\in \ob{\E}}$ (for example, an arbitrary Boolean Grothendieck topos). Then the following three concepts correspond bijectively.
\begin{enumerate}
    \item quotients of $\E$
    \item internal filters of $\Xi$
    \item internal semilattice homomorphisms $\Xi\to \Omega$
\end{enumerate}
\end{corollary}

%d%r: main corはこっちを引用すべきでは?
It is worth emphasizing that it gives a partial solution to Lawvere's open problem \cite{OpenLawvere}:
\begin{corollary}[Lawvere's open problem for Boolean toposes]\label{mainCor}
For a locally small Boolean topos with a local state classifier (for example, an arbitrary Boolean
Grothendieck topos), there exists an internal parameterization of quotients. In particular, the number of quotients is small.
\end{corollary}
\begin{proof}
It is implied by the above argument and Corollary \ref{CorNumberofHyperQuotients}.
\end{proof}

As a toy example of classification of all quotients of a Boolean topos, using Corollary \ref{CorInternalparameterizationOfQuotientsOfBooleanTopos}, we classify all quotients of the topos of species.
\begin{example}[Quotients of the topos of species]
In Example \ref{exampleSpecies}, we see that a local state classifier of the topos of species $\Species$ is given by $\Xi = \SubGrp{\Aut{-}}$, which we call \dq{the species of symmetries.}

An internal filter of $\Xi$ is given by a family of filters
\[\{F_A \subset \SubGrp{\Aut{A}}\}_{A \in \ob{\FinSet_{0}}}\]
that is closed under the action of bijections. It is equivalent to saying that the family of filters \[\{F_n \subset \SubGrp{\symG{n}}\}_{n=0}^{\infty}\] that is closed under conjugate actions for each $0 \leq n < \infty$. Furthermore, since a filter of finite semilattice is principal, it corresponds to a family of normal subgroups 
\[
\{N_n \subset \symG{n}\}_{n=0}^{\infty}.
\]
Therefore, there is a natural bijection between quotients of the topos of species and families of normal subgroups $\{N_n \subset \symG{n}\}_{n=0}^{\infty}$ of all symmetric groups.

Classification of normal subgroups of symmetric groups is well-known. Since an alternating group $A_n$ of a symmetric group $\symG{n}$ is simple when $n \geq 5$, there exist exactly $3$ different normal subgroups for $\symG{n} (n \geq 5)$, namely two trivial subgroups and the alternating group. For the cases of $n=0, 1, 2, 3, 4$, by concrete calculation, we can prove that there are $1, 1, 2, 3, 4$ subgroups, respectively. (There is a non-trivial normal subgroup of $\symG{4}$ that differs from $A_4$, which is the Klein four-group.) In particular, the number of quotients of the topos of species is the cardinality of the
% following infinite product,
% \[\{1 = A_0 =\symG{0}\}\times\{1 = A_1 =\symG{1}\}\times\{1 = A_2 ,\symG{2}\}\times\{1, A_3,\symG{3}\}\times\{1, V, A_4,\symG{4}\}\times\{1, A_5,\symG{5}\}\times\dots,\]
continuum.
\end{example}

See the next subsection for another example of the classification of all quotients of a Boolean topos via Corollary \ref{CorInternalparameterizationOfQuotientsOfBooleanTopos}.

\subsection{Topos of (topological) group actions}\label{ExampleofGroupTopos}
As mentioned in Example \ref{HQfromTopMonoid}, for a topological group $(G,\mathcal{\tau})$, the continuous action topos $\Cont{G}{\tau}$ gives a (hyperconnected) quotient of $\ps{G}$. In this subsection, we prove the converse statement (Corollary \ref{CorConnectedGroupAction}), as a corollary of our main theorem, although an essentially same argument is already known \cite[][section 5.3.1]{rogers2021toposesM}, for this particular example.

Just to avoid the argument becoming wordy, we introduce two terminologies.
\begin{definition}[Conjugate closed filter]\label{DefConjClosedFilter}
A filter $F$ of a (semi)lattice of subgroups $\SubGrp{G}$ is called \emph{conjugate closed} if it is closed under the conjugate action.
\end{definition}
\begin{definition}[Simple topological group]\label{DefSimpleTopGrp}
A topological group structure on a group $G$ is called \emph{simple} if the set of all open subgroups is a fundamental system of neighborhoods of the identity element.
\end{definition}

\begin{example}[Simple topological group and non-simple topological group]
The topological group of $p$-adic numbers $\Z_p$ with the usual $p$-adic topology is simple. In contrast to that, the topological group of real numbers $\R$ with the usual Euclidean topology is not simple.
\end{example}

In Example \ref{ExampleOfLSCOfGroupActionTopos}, we have seen that a local state classifier of a group action topos $\ps{G}$ is given by the set of all subgroups of $G$, equipped with morphisms $\xi_X$ that send an element to its stabilizer subgroup. Its internal filter is a conjugate closed filter $\F \subset \SubGrp{G}$. As a corollary of our main theorem, we have the following:
\begin{corollary}
For a group $G$, quotients of a group action topos $\ps{G}$ bijectively correspond to conjugate closed filters of $\SubGrp{G}$.
\end{corollary}
\begin{proof}
This is immediately corollary of our main theorem and Proposition \ref{PropQuotientofBooleanTopos}.
\end{proof}

To connect this corollary to the notion of a topological group, we prove the following lemma.

\begin{lemma}\label{LemmaTopStrAndFilters}
For a topological group $(G,\tau)$, the set of all open subgroups is a conjugate closed filter of $\SubGrp{G}$. Furthermore, for a group $G$, this construction gives a bijective correspondence between simple topological group structures on $G$ and conjugate closed filters of $\SubGrp{G}$.
% and its conjugate-closed filter $\F$, there uniquely exists a topological group structure on $G$ such that its set of all open subgroups is $\F$. 
\end{lemma}
\begin{proof}
% The former statement is easily proved by the fact that the multiplication map $g \cdot \colon G\to G $ for each $g \in G$ is homeomorphic.
For a topological group $(G,\tau)$, let $F_{G,\tau}$ denote the set of all open subgroups of $(G,\tau)$.

First, we prove the former statement. Take an arbitrary topological group $(G,\tau)$. Since the intersection of finitely many open subgroups (including the nullary intersection $G$) is an open subgroup, the set $F_{G,\tau}$ is closed under finite meet. If an open subgroup $H\in F_{G,\tau}$ is a subset of another subgroup $H' \in \SubGrp{G}$, $H'$ is a union of some cosets of $H$ and is open. Hence the filter $F_{G,\tau}$ is upward closed. Since a conjugate subgroup of an open subgroup is also open, the set $F_{G,\tau}$ is a conjugate-closed filter. The former statement is now proven.

Second, we prove the latter statement. If two topological group structure $(G, \tau)$ and $(G,\tau')$ on a group $G$ are both simple and $F_{G,\tau}= F_{G,\tau'}$, then two topologies coincide $\tau = \tau'$. Therefore, it is enough to construct, for an arbitrary given conjugate closed filter $F$, a simple topological group structure $(G,\tau)$ such that $F_{G,\tau}= F$. Take an arbitrary conjugate closed filter $F$. We define a set $B$ of subsets of $G$, as follows:
\[B\defeq \{gHg'\mid g,g' \in G, \ H \in F\}.\]
Then the set $B$ has a slightly simpler description:
\[B= \{gH\mid g\in G, \ H \in F\}=\{Hg\mid g\in G, \ H \in F\},\]
because $F$ is conjugate closed and we have an equation $gHg' = gg'(g'^{-1}Hg')= (gHg^{-1})gg'$.
Furthermore, if $gH \in B$ contains $x \in G$, we can prove that $gH=xH$, then $\{S\in B\mid x\in S \}= \{xH\mid H\in F\}$.

We define $\tau$ as a topology on $G$ that is generated by open basis $B$. In other words, an open set of $\tau$ is a union of some elements of $F$. But, this construction needs verification. First, since $G\in F$, $F$ covers whole $G$. Second, take $x \in G$ and $gH,g'H' \in B$ such that $x \in gH$ and $x \in g'H'$. Since
\[
x \in gH \cap g'H' = xH \cap xH' = x(H\cap H')\in B
\]
and $x(H\cap H') \subset gH,g'H'$, the verification is done. Let $\tau$ denote the topology we have just verified. Notice that $\{xH\mid H\in F\}$ is a fundamental system of neighborhoods of $x\in G$.

We prove that $(G,\tau)$ is a topological group, is simple, and $F_{G,\tau}=F$.

Since $(gH)^{-1}=Hg^{-1}$, inverse element function $(-)^{-1}\colon G \to G$ is continuous. To prove the multiplication map $\ast\colon G\times G \to G$ is continuous, take $x,y\in G$ and $xyH \in B$. $x(yHy^{-1})$ and $yH$ are open neighborhoods of $x,y$ and $x(yHy^{-1}) \times yH \subset G\times G$ is sent by $\ast$ to $xyH\subset G$. Then, $\ast$ is continuous and $(G,\tau)$ is a topological group.

Now, it is enough to prove the equation $F_{G,\tau}=F$, because the simpleness of the topological group $(G,\tau)$ is immediately implied by this equation. Since every element of $F$ is an open subgroup of $(G,\tau)$ (i.e., $F_{G,\tau}\supset F$), we prove the converse inclusion relation, $F_{G,\tau}\subset F$. Take an arbitrary open subgroup $U\in F_{G,\tau}$. Since $U$ is open and $F$ is a fundamental system of neighborhoods of the identity element, there exists $H \in F$ such that $H \subset U$. Because $F$ is upward closed, the open subgroup $U$ also belongs to $F$, and thus our proof is completed.
\end{proof}

By the above argument, we have the next corollary.
\begin{corollary}[Quotients of a group action topos]\label{CorConnectedGroupAction}
% Let $G$ be a group. All quotients of $\ps{G}$ are given by a topological group structure on $G$.
For a group $G$, the following three concepts correspond bijectively:
\begin{enumerate}
    \item (hyper)connected quotients of the topos of group action $\ps{G}$,
    \item conjugate closed filters (Definition \ref{DefConjClosedFilter}) of $\SubGrp{G}$,
    \item simple topological group structure (Definition \ref{DefSimpleTopGrp}) on $G$.
\end{enumerate}
In particular, every quotient of the topos of group actions is induced by a (simple) topological group structure on $G$.
\end{corollary}
\begin{proof}
The former statement is implied by the above arguments. The latter is verified by concretely checking the correspondence. 
\end{proof}

We end this subsection by mentioning two possible generalizations from groups to topological groups and to monoids.
\begin{remark}[Quotients of a topological group action topos]
How about quotients of a topos of topological group actions, which is also a typical example of a Boolean Grothendieck topos? Actually, we have already essentially completed the classification of all quotients of a topological group action topos. It is because, a topological group action topos is a quotient of a (discrete) group action topos, and  a quotient of a quotient is itself a quotient. 

Specifically, consider the topos $\Cont{G}{\tau}$ of continuous actions of the topological group $(G,\tau)$. Since this topos is also a Boolean Grothendieck topos, we can apply Proposition \ref{PropQuotientofBooleanTopos}. Then, by Lemma \ref{LemmaTopStrAndFilters}, we can assume that it is simple. Then we can conclude that a quotient of $\Cont{G}{\tau}$ corresponds bijectively to a simple topological group structure equal to or weaker than $\tau$. 
From this generalized point of view, we can regard Corollary \ref{CorConnectedGroupAction} as the particular case where $\tau$ is the discrete topology.
\end{remark}

\begin{remark}[Hyperconnected quotients of a topos of monoid actions]
In this remark, we briefly mention the extent to which the discussion in this subsection can be generalized from groups to monoids.

There are at least two differences. First, a monoid action topos is not necessarily Boolean, so we cannot apply Proposition \ref{PropQuotientofBooleanTopos}. In other words, some quotients may not be hyperconnected. The other difference is that a local state classifier $\Xi$ is not the set of all subalgebras (in this case, submonoids), unlike the case of a group. A local state classifier, calculated by Example \ref{LSCofPresheafTopos}, is the set of all equivalence relations on $M$ that are compatible with the right action of $M$ on $M$ itself. In other words, it is the set of \emph{right congruences} of $M$.

In \cite[][section 5.3.1]{rogers2021toposesM}, the classification of hyperconnected quotients of the topos of monoid actions is given in terms of right congruences, in essentially the same way that our main theorem provides for this particular case. Furthermore, in the same paper, a hyperconnected quotient of a topos of monoid actions that is not induced by the topological monoid structure is given, using the additive monoid of natural numbers $(\mathbb{N},+,0)$.
%topos of endofunctionsと同値なの，言わんでええやろ．
\end{remark}

% \subsection{The topos of monoid actions}\label{ExampleofMonoidTopos}
% cite Morgan Rogers
% the topos of endofunctions
%o%r: なくてもいっか?
% \subsection{Hyperconnected quotient induced by a full bijective on objects functors}

\section{Conclusions and future works}\label{secConclusion}
In this paper, we defined the notion of a local state classifier (Definition \ref{DefLocalStateClassifier}). By making use of it, we obtained the internal parameterization of hyperconnected quotients (Theorem \ref{mainTheorem}). It allows us to classify all hyperconnected quotients by calculating one object, a local state classifier. After proving the main theorem, we demonstrated the classifications for some familiar toposes. As a corollary of our main theorem, we gave a partial solution to Lawvere's open problem \cite{OpenLawvere}. This is a novel step towards the solution of the open problem, especially from the perspective of internal parameterization. 

However, there are still many remaining things to do. We list some of them, including some that are vague, as future works.

First, seek the possible generalization of our internal parameterization to all quotients, not limited to hyperconnected quotients, to solve Lawvere's first open problem \cite{OpenLawvere}. However, it seems impossible to realize such a generalization by manipulating only a local state classifier. It is because a local state classifier seems too small to parameterize all quotients. For example, in the case of $\DirGraph$, its local state classifier and subobject classifier are both finite graphs. However, the number of quotients of $\DirGraph \simeq \ps{\Par}$ is at least the cardinality of the continuum. (This lower bound is given by the fact that $\ps{\Z}$ is realized as a quotient of $\DirGraph$. It is because there is a lax epimorphism $\Par \to \Z$ from the parallel morphism category $\Par$ (Example \ref{ExampleLSCofDirectedGraph}) to the group $\Z$ in the $2$-category of categories $\mathrm{Cat}$. See \cite{adamek2001functors} or \cite{el2002simultaneously} for lax epimorphism.) 
% In this particular example, the local state classifier might be too small to parameterize all quotients internally. 
Therefore, an essential modification may be needed if one seeks a similar internal parameterization of all quotients.

Second, find other applications of a local state classifier in other contexts of other categories, not limited to toposes. Although the definition of a local state classifier does make sense in other categories, we do not yet know other applications of a local state classifier. One of the obstructions is that a local state classifier tends to be the terminal object (or degenerate) in some familiar categories, as we see in Example \ref{ExampleTerminalLSC} and Example \ref{ExampleexistenceofLSCofEquational}. To avoid this problem, some variances of a local state classifier might work. There are many options to define variances of a local state classifier. For example, we can define a colimit of all regular or split monomorphisms instead of all monomorphisms or define algebraic structure on a (variance of) local state classifier $\Xi$ by a monoidal structure, not only by a cartesian structure as done in subsection \ref{subsectionSemilatticeStructure}.

Third, find a class of functors that preserves a local state classifier. One fundamental method to study categorical structures is considering the preservation of the structure by functors. However, we do not find a suitable class of functors that preserves a local state classifier. Though we define a local state classifier as a colimit, even a cocontinuous functor may not preserve it since the indexing category of the diagram $\mo{\C}$ is possibly large and depends on the considered category $\C$. As seen in Remark \ref{RemarkSheafificarionandLSC}, even a sheafification functor (which is an essentially surjective left exact left adjoint functor) may not preserve it.
%o%m: Hyperconnectedのdirect imageは保ちそうじゃね?

Fourth, clarify the relationship to classifications of smaller classes of quotients. Some classifications of smaller class quotients are known, including the classification of atomic quotients mentioned in Example \ref{ExampleAtomicConnected}. We want to clarify how a restricted class of hyperconnected quotients can be described in terms of a local state classifier.

Lastly, study the interaction with other internal structures. One crucial point of this paper is to find the internal structures that correspond to the external structures, hyperconnected quotients. Once we succeed in internalizing, it is natural to consider the relationship with other internal structures. For example, since a Lawvere-Tierney topology is an idempotent internal semilattice homomorphism on the subobject classifier $\Omega$, from an internal semilattice homomorphism $\Xi \to \Omega$, we can obtain a new homomorphism $\Xi \to \Omega$ just by composition. In terms of corresponding external structures, if we have a subtopos and a hyperconnected quotient of a given topos (with a local state classifier), we can obtain a new hyperconnected quotient. This fact is not apparent without our internal parameterization. More generally, the monoid of all internal semilattice endo-homomorphisms $\Omega\to \Omega$ (which is called \emph{productive weak Lawvere-Tierney topology} in \cite{khanjanzadeh2021weak}) naturally acts on the set of all hyperconnected quotients.

\section*{Acknowledgements}
The author would like to thank Ryu Hasegawa and Hisashi Aratake for their suggestions and support, and Yuta Yamamoto and Haruya Minoura for their useful discussions. This research was supported by Forefront Physics and Mathematics Program to Drive Transformation (FoPM), a World-leading Innovative Graduate Study (WINGS) Program, the University of Tokyo.
%先生，荒武さん， FoPM．
\appendix
\section{Internal semilattices and the universal filter}\label{AppendixInternalSemilattice}
In this appendix, we briefly recall the notions of internal semilattices (Definition \ref{DefOfBSL}) and their filters (Definition \ref{InternalFilter}) and prove that the subobject classifier of a topos is universal among internal filters (Proposition \ref{PropUniversalityOfSubobjectClassifierAsSemilattice}).

First, we define the equational theory of (bounded meet-)semilattice, instead of usual semilattices (in $\Set$), since we interpret it in other cartesian categories (like toposes). For internal interpretation of an equational theory, see \cite[][section IV.8.]{maclane2012sheaves}. 

% In this paper, the word \dq{bounded} means the existence of the maximum element $\top$.
\begin{definition}[Equational theory of semilattice]\label{DefOfBSL}
The equational theory of \emph{semilattice} consists of two operations $\top, \meet$, which have arity $0,2$ respectively, and four axioms
\begin{itemize}
    \item $x\meet \top = \top \meet x=x$
    \item $(x\meet y)\meet z=x\meet (y\meet z)$
    \item $x\meet y=y\meet x$
    \item $x\meet x = x$.
\end{itemize}
\end{definition}

The theory of semilattice is the same thing as the theory of idempotent commutative monoid. However, in this paper, we prefer calling it a semilattice in order to emphasize its order structure.

Internal filters of an internal semilattice can be defined as a \dq{upward closed} subalgebra.
\begin{definition}[Internal filter]\label{InternalFilter}
Let $\C$ be a finitely complete category and $(X,\top,\meet)$ be an internal semilattice. \emph{An internal filter} of $(X,\top,\meet)$ is a subobject $m\colon F\subob X$ that satisfies the following three conditions.
\begin{enumerate}
    \item (closed under $\top$) There is a (necessarily unique) morphism $\top_{F}\colon \1\to F$ such that 
    \[
    \begin{tikzcd}
    \1\ar[r,dotted,"\top_{F}"]\ar[d,"\id{\1}"]&F\ar[d,\mono, "m"]\\
    \1\ar[r,"\top"]&X\\
    \end{tikzcd}
    \]
    commutes.
    
    \item (closed under $\meet$) There is a (necessarily unique) morphism $\meet_{F}\colon F\times F\to F$ such that 
    \[
    \begin{tikzcd}
    F\times F\ar[r,dotted,"\meet_{F}"]\ar[d,\mono,"m\times m"]&F\ar[d,\mono,"m"]\\
    X\times X\ar[r,"\meet"]&X\\
    \end{tikzcd}
    \]
    commutes.
    
    \item (upward closed) For the equalizer $k\colon \Eq \subob X$ of
    \[
    \begin{tikzcd}[row sep=tiny]
        &           &X\times X\ar[rd,"\meet"]&\\
    \Eq \ar[r,"k",\mono]&F\times X\ar[ru,"m\times \id{X}",\mono]\ar[rd,"\mathrm{pr}_1"']&&X\\
        &           &F\ar[ru,"m"',\mono]&,
    \end{tikzcd}
    \] there is a (necessarily unique) morphism $l\colon \Eq\to F\times F$ such that
    \[
    \begin{tikzcd}
    \Eq\ar[r,"k",\mono]\ar[d,dotted,"l"']&F\times X\\
    F\times F\ar[ru,"\id{F}\times m"',\mono]&
    \end{tikzcd}
    \]
    commutes.
\end{enumerate}
\end{definition}
The third condition above is just a diagrammatic version of 
\[\forall (u,x)\in F\times X, (u\meet x=u \implies (u,x) \in F\times F),\]
which means that $F$ is upward closed.

\begin{example}[$\Set$]
In the category of sets $\Set$, an internal semilattice is a usual semilattice. It is the same as a poset with finite meet, including the maximum element ($=$ the nullary meet). An internal filter is a filter in the usual sense, which is an upward closed subset that is closed under finite meet. Note that the existence of the nullary meet ensures that a filter is non-empty.
\end{example}

\begin{example}[Presheaf topos]
In a presheaf topos $\ps{\C}$ of a small category $\C$, an internal semilattice structure on a presheaf $P\colon \C^{\op}\to \Set$ is a family of semilattice structures on $Pc$ for each $c\in \ob{\C}$ such that $Pf\colon Pc'\to Pc$ for each $f\colon c\to c'$ is a semilattice homomorphism. An internal filter of $P$ is a subpresheaf $F\subob P $ such that $Fc$ is a filter of $Pc$ in the usual sense for each $c\in \ob{C}$.
\end{example}

Like the case of a presheaf topos, if $\C$ is locally small, an internal semilattice and an internal filter can be described \emph{externally} with the Yoneda lemma and reduced to the theory of usual semilattices and filters. 

\begin{example}[Subobject classifier $\Omega$]
Recall that a subobject classifier $\Omega$ in an elementary topos $\E$ has the canonical internal Heyting algebra structure (\parencite[see][IV.8. Theorem 1]{maclane2012sheaves}). By restricting that structure, $\Omega$ has the canonical internal semilattice structure. It is familiar in the context of Lawvere-Tierney topology to regard $\Omega$ as an internal semilattice rather than internal Heyting algebra. In fact, a Lawvere-Tierney topology is an idempotent internal semilattice endomorphism of $\Omega$, not an internal Heyting algebra endomorphism. The (universal) subobject \[\true\colon \1\subob \Omega\] is an internal filter of $\Omega$.
\end{example}

From now, we prove that for an arbitrary elementary topos $\E$, the subobject classifier $\true\colon \1\subob \Omega$ is the universal internal filter. To prove that, we first prepare the following two lemmas.

\begin{lemma}\label{pboffilt}
Let $\C$ be a finitely complete category, $X$ be an internal semilattice, and $m\colon F\subob X$ be its internal filter. Then, the following two diagram 
    \[
    \begin{tikzcd}
    \1\ar[r,"\top_{F}"]\ar[d,"\id{\1}"]&F\ar[d,\mono, "m"]\\
    \1\ar[r,"\top"]&X\\
    \end{tikzcd}
    \]
    and
    \[
    \begin{tikzcd}
    F\times F\ar[r,"\meet_{F}"]\ar[d,\mono,"m\times m"]&F\ar[d,\mono,"m"]\\
    X\times X\ar[r,"\meet"]&X,
    \end{tikzcd}
    \]
    which are in the Definition \ref{InternalFilter}, are both pullback diagrams.
\end{lemma}
\begin{proof}
The first diagram can be divided into two elementary pullback diagrams:
\[
\begin{tikzcd}
\1\ar[d,"\id{\1}"]\ar[r,"\top_{F}"]\arrow[dr, phantom,""{pullback}, very near start]
&F\ar[d,"\id{F}"]\ar[r,"\id{F}"]\arrow[dr, phantom,""{pullback}, very near start]
&F\ar[d,\mono,"m"]\\
\1\ar[r,"\top_{F}"]\ar[rr,bend right,"\top"']&
F\ar[r,"m",\mono]&X.
\end{tikzcd}
\]

The statement that the second diagram is pullback is the diagrammatic version of the formula
\[\forall (x,x')\in X\times X,\  x\meet x' \in F \implies (x\in F \ \text{and}\  x'\in F).\] If $\C$ is locally small, by the Yoneda lemma, it is enough to show this in $\Set$. In $\Set$, the desired statement can be easily verified by the third condition of internal filters, upward closed. Even if $\C$ is not locally small, this can be proved in a straightforward diagram chase or the argument of generalized elements.
\end{proof}

\begin{lemma}[Pullback-stability of filters]\label{PullbackStabilityOfFilters}
Let $\C$ be a finitely complete category, $X,Y$ be internal semilattices in $\C$, $f\colon X\to Y$ be an internal semilattice homomorphism and $m\colon F\subob Y$ be an internal filter of $Y$. Then the pullback $f^{\ast}m\colon f^{*}F\subob X$ of $m\colon F\subob X$ along $f$
\[
\begin{tikzcd}
f^{*}F \ar[r]\ar[d,\mono,"f^{\ast}m"']\arrow[dr, phantom,""{pullback}, very near start]&F\ar[d,"m",\mono]\\
X\ar[r,"f"]&Y\\
\end{tikzcd}
\]
is an internal filter of $X$.
\end{lemma}
\begin{proof}
If $\C$ is locally small, by using the Yoneda lemma, it is enough to show this lemma only for $\C=\Set$. In $\Set$, an inverse image of a filter via a semilattice homomorphism is again a filter. For the case that $\C$ is not locally small, this lemma can also be easily proved by a straightforward diagram chase.
\end{proof}

% According to this lemma, we can define a contravariant functor of filters $\mathrm{Filters}\colon \mathrm{BSL}(\C)^{\op}\to \Set$ for a locally small $\C$.
We move on to the proof of the universal property of the subobject classifier as an internal filter.
\begin{proposition}[Subobject classifier is the universal internal filter]\label{PropUniversalityOfSubobjectClassifierAsSemilattice}
Let $\E$ be a topos, $\true\colon \1\to\Omega$ be its subobject classifier, $X$ be an internal semilattice, and $m\colon F\subob X$ be a subobject of $X$. Then, $m\colon F\subob X$ is an internal filter if and only if the characteristic map $\ch{F}\colon X\to \Omega$ is an internal semilattice homomorphism.
\[
\begin{tikzcd}
F \ar[r]\ar[d,\mono,"m"']\arrow[dr, phantom,""{pullback}, very near start]&\1\ar[d,"\true",\mono]\\
X\ar[r,"\ch{F}"]&\Omega
\end{tikzcd}
\]
Consequently, there is a bijective correspondence between internal semilattice homomorphisms $X\to \Omega$ and internal filters of $X$.
\end{proposition}
\begin{proof}
By Lemma \ref{PullbackStabilityOfFilters}, it is enough to prove that $\ch{F}$ is an internal semilattice homomorphism for any internal filter $m\colon F\subob X$.

First, we show that $\ch{F}$ preserves $\top$, i.e.,
\[
\begin{tikzcd}[row sep =tiny]
&X\ar[dd,"\ch{F}"]\\
\1\ar[ru,"\top"]\ar[rd,"\true"']&\\
&\Omega\\
\end{tikzcd}
\]
commutes. To prove this, it is enough to observe that the corresponding subobject of $\ch{F}\circ \top$ is equal to $\id{\1}\colon \1\subob \1$. 
Since $m\colon F\subob X$ is an internal filter, by Lemma \ref{pboffilt} and the definition of $\ch{F}$, we have the following pullback diagram
\[
\begin{tikzcd}
\1\ar[d,"\id{\1}"]\ar[r,"\top_{F}"]\arrow[dr, phantom,""{pullback}, very near start]
&F\ar[d,\mono,"m"]\ar[r,"\excl"]\arrow[dr, phantom,""{pullback}, very near start]
&\1\ar[d,\mono,"\true"] \\
\1\ar[r,"\top"']&
X\ar[r,"\ch{F}"]&
\Omega.
\end{tikzcd}
\]
The composite pullback diagram is what we needed.

Second, we show that $\ch{F}$ preserves $\meet$, i.e.,
\[
\begin{tikzcd}
X\times X\ar[d,"\ch{F}\times \ch{F}"]\ar[r,"\meet"]&X\ar[d,"\ch{F}"]\\
\Omega \times \Omega \ar[r,"\meet"]& \Omega
\end{tikzcd}
\]
commutes. Again by Lemma \ref{pboffilt}, we obtain the following two pullback diagrams:
\[
\begin{tikzcd}
F\times F\ar[r,"\excl\times \excl"]\ar[d,"m\times m",\mono]\arrow[dr, phantom,""{pullback}, very near start]&
\1\times \1\ar[r,"\meet_{\1}(=\excl)"]\ar[d,"\true\times \true",\mono]\arrow[dr, phantom,""{pullback}, very near start]&
\1\ar[d,"\true",\mono]\\
X\times X\ar[r,"\ch{F}\times \ch{F}"]&\Omega \times \Omega\ar[r,"\meet"]&\Omega
\end{tikzcd}
\]
and
\[
\begin{tikzcd}
F\times F\ar[r,"\meet_{F}"]\ar[d,\mono,"m\times m"]\arrow[dr, phantom,""{pullback}, very near start]&
F\ar[r,"\excl"]\ar[d,\mono,"m"]\arrow[dr, phantom,""{pullback}, very near start]&
\1\ar[d,\mono,"\true"]\\
X\times X\ar[r,"\meet"]&X\ar[r,"\ch{F}"]&\Omega.
\end{tikzcd}
\] This shows that $\meet\circ (\ch{F}\times \ch{F})$ and $\ch{F}\circ \meet$ are the characteristic map of the same subobject $m\times m\colon F\times F \subob X\times X$.
\end{proof}

We conclude this appendix by rephrasing Proposition \ref{PropUniversalityOfSubobjectClassifierAsSemilattice} in terms of the representability of a functor. For a locally small topos $\E$, there is a contravariant functor from the category of internal semilattices in $\E$ to the category of sets $\Set$ that sends an internal semilattice $X$ to the set of all internal filters of $X$. The subobject classifier $\Omega$ represents this functor, with the universal element 
\[\true\colon \1\subob \Omega .\]
% Thinking internal filters is as natural as thinking semilattices and $\true$

\section{Existence theorem for a local state classifier}\label{AppendixExistenceTheorem}
%d%Local state classifierを，ay(c)のquotientの為すpresheafのsheafificationとして実現して，このsubsectionを置き換えたい．
In this appendix, we show an existence theorem for a local state classifier. The theorem is strong enough to prove the existence of a local state classifier not only for a Grothendieck topos, which is concretely constructed in subsection \ref{subsectionLSCofGrothendiecktopos}, but also for a category of models of an equational theory (Example \ref{ExampleexistenceofLSCofEquational}).

Our method is similar to the proof of the general adjoint functor theorem. We reduce a local state classifier to a colimit of a small diagram under the assumption that there exists a set of objects that \dq{generates} the category in some sense. To make it precise, we first define the notion of \dq{mono-density}.

First, recall the notion of a dense full subcategory. Let $\Sigma \subset \ob{\C}$ be a set of objects of a category $\C$. By abuse of language, let $\Sigma$ also denote a full subcategory of $\C$ consisting of objects in $\Sigma$. Then, the full subcategory $\Sigma$ is said to be dense, if the following diagram
\[
\begin{tikzcd}
\Sigma\ar[rr,\mono,""'{name=A}]\ar[rd,\mono]&&\C\\
&\C\ar[ru,"\id{\C}"']\ar[Leftarrow, "\id{}",to = A]&
\end{tikzcd}
\]
is a pointwise left Kan extension.

Our assumption for the existence theorem, \dq{mono-density,} is a variant of this density, regarding monomorphisms. In the following definition, $\mo{\Sigma}$ denotes the full subcategory of $\mo{\C}$ spanned by objects in $\Sigma$.
\begin{definition}(Mono-density)\label{DefinitionMonoDensity}
    For a category $\C$ and a set of objects $\Sigma \subset \ob{\C}$ is \emph{mono-dense} if the following diagram
\[
\begin{tikzcd}
\mo{\Sigma}\ar[rr,\mono,""'{name=A}]\ar[rd,\mono]&&\C\\
&\mo{\C}\ar[ru,\mono]\ar[Leftarrow, "\id{}",to = A]&,
\end{tikzcd}
\]
(where all functors above are the canonical inclusions) is a pointwise left Kan extension.
\end{definition}

In other words, $\Sigma$ is mono-dense if and only if, for any object $X\in \ob{\C}$, $X$ is the colimit of the following functor
\[
\begin{tikzcd}
\comma{\mo{\Sigma}}{X}\ar[r]&\mo{\Sigma}\ar[r,\mono]&\C
\end{tikzcd}
\]
with the canonical cocone, where the domain $\comma{\mo{\Sigma}}{X}$ denotes the comma category of 
\[
\begin{tikzcd}
\mo{\Sigma}\ar[r,\mono]&\mo{\C}&1\ar[l,"X"'].
\end{tikzcd}
\]
Since $\comma{\mo{\Sigma}}{X}$ is equivalent to the poset of subobjects of $X$ that belong to $\Sigma$, mono-density of $\Sigma$ means that every object of $\C$ is the colimit of all subobjects in $\Sigma$, in a canonical way.

%d%f: Theorem(uouo) とかのuouoのフォーマット揃える．冠詞なしの大文字スタートにした．
%d% densityのsubob版であることを明記する．
%d% locally smallをを入れて．でないとsmallになるとは限らないので．するとpointwiseにもなるのでそこも訂正．
\begin{proposition}[Existence theorem]\label{ExistenceTheorem}
If a locally small and cocomplete category $\C$ has a mono-dense (small) set of objects $\Sigma \subset \ob{\C}$, then $\C$ has a local state classifier.
\end{proposition}
\begin{proof}
Since $\C$ is locally small and $\Sigma$ is a (small) set, the cocompleteness of $\C$ implies that there is the following left Kan extension ($=$ the colimit of $\mo{\Sigma}\subob \C$),
\[
\begin{tikzcd}
\mo{\Sigma}\ar[rr,\mono,""'{name=A}]\ar[rd,"\excl"']&&\C\\
&1\ar[ru,"\Xi"']\ar[Leftarrow, "\tau",to = A]&.
\end{tikzcd}
\]

By the mono-density of $\Sigma$ and the universal property of the left Kan extension, the above diagram, which is equal to 
\[
\begin{tikzcd}
\mo{\Sigma}\ar[rrr,\mono,""'{name=A}]\ar[rd,\mono]&&&\C\\
&\mo{\C}\ar[rd,"\excl"']&&\\
&&1\ar[ruu,"\Xi"']\ar[Leftarrow, "\tau",to = A]&
\end{tikzcd}
\]
is uniquely factored as
\[
\begin{tikzcd}
\mo{\Sigma}\ar[rrr,\mono,""'{name=A}]\ar[rd,\mono]&&&\C\\
&\mo{\C}\ar[rd,"\excl"']\ar[rru,\mono, ""'{name=B}]\ar[Leftarrow, "\id{}",to = A]&&\\
&&1\ar[ruu,"\Xi"']\ar[Leftarrow, "\exists \excl \ \xi", to = B]&.
\end{tikzcd}
\]

Then,  
\[
\begin{tikzcd}
\mo{\C}\ar[rr,\mono,""'{name=A}]\ar[rd,"\excl"']&&\C\\
&1\ar[ru,"\Xi"']\ar[Leftarrow, "\xi",to = A]&
\end{tikzcd}
\]
defined as above, is a left Kan extension, and gives a local state classifier of $\C$.
\end{proof}

\begin{example}(Grothendieck topos)
Now we obtain another proof of the existence of a local state classifier of a Grothendieck topos, which is already proved in subsection \ref{subsectionLSCofGrothendiecktopos}.
For a small site $(\C, J)$, we say a $J$-sheaf $P$ is \emph{generated by one element}, if there is an object $c\in \ob{C}$ and an epimorphism $\ay{c}\quo P$ in $\Sh{\C,J}$. 
% Since the isomorphic classes of $J$-sheaves thet generated by one element is
We define $\Sigma$ as the (essentially small) set of all $J$-sheaves generated by one element. 
% It is enough to show that $\Sigma$ satisfies the hypothesis of proposition \ref{ExistenceTheorem}. 

The set $\Sigma$ is mono-dense because of Lemma \ref{LemmaGenerateSubsheaf}. (In order to check this, one can use the fact that the (fully faithful) inclusion functor $\Sh{C,J}\to \ps{\C}$ reflects all colimits.) Therefore, we can apply Proposition \ref{ExistenceTheorem} to $\Sh{C,J}$ and obtain a local state classifier.
\end{example}

\begin{example}[Category of models of an equational theory]\label{ExampleexistenceofLSCofEquational}
For any equational theory $T$, the category of $T$-algebras, $T$-$\mathrm{Alg}$, has a local state classifier. It is because $T$-$\mathrm{Alg}$ is locally small and cocomplete (for example, see \cite[][COROLLARY 5.6.14.]{riehl2017category} or \cite[][ section 3.4]{borceux1994handbook}), the set $\Sigma$ of all finitely generated $T$-algebras is essentially small, and every $T$-algebra is the colimit of its finitely generated subalgebras in a canonical way. This example includes Example \ref{ExampleOfLSCOfGroupActionTopos} and Example \ref{ExampleLSCofPointedSets}. 

However, a local state classifier of a category of familiar algebras tends to be a terminal object. For example, local state classifiers of the category of groups $\mathrm{Group}$, the category of rings $\mathrm{Ring}$, the category of commutative rings $\mathrm{CRing}$, the category of monoids $\mathrm{Monoid}$, the category of abelian groups $\mathrm{Ab}$, the category of lattices $\mathrm{Lattice}$, and the category of vector spaces $\mathrm{Vect}_{\mathbb{K}}$ are terminal.
%o%m: semilatticeについて触れる? Non-terminalな気がする．
\end{example}

% Now we move on to the proof that every Grothendieck topos has the local state classifier. Let $\ay{c}$ denote the sheafification of $\C(-,c)$.
% \begin{lemma}\label{OGsheaf}
% Let $(\C,J)$ be a site, $P\colon \C^{\op}\to \Set$ be a $J$-sheaf and $x\in Pc$ be its element. Then there exists the minimum sub-$J$-sheaf $\gen{x} \subob P$ that contains $x$.
% \end{lemma}
% \begin{proof}
%  The image of the morphism $\ay{c}\to P$ that corresponds to $x\in Pc$ by the Yoneda lemma for sheaves \cite[][III.6.(17)]{maclane2012sheaves} is the minimum sub-$J$-sheaf $\gen{x} \subob P$ that contains $x$.
% \end{proof}

% This existence theorem does not merely state the existence of the local state classifier, but explicitly teaches the diagram for the small colimit that gives the local state classifier. In fact, for some Grothendieck topos including the following example, it is practical to compute the local state classifier in this way.

% \begin{example}[the local state classifier of localic toposes]\label{LSCofLocalic}
% The local state classifier of a localic Grothendieck topos is the terminal object, since $\Sigma$ in the proof of Corollary \ref{LSCofGT} consists of all subterminals. See corollary \ref{CorLocalicIFFterminal} for the converse statement.
% \end{example}

\printbibliography

@book{maclane2012sheaves,
  title={Sheaves in geometry and logic: A first introduction to topos theory},
  author={MacLane, Saunders and Moerdijk, Ieke},
  year={2012},
  publisher={Springer Science \& Business Media}
}

@book{riehl2017category,
  title={Category theory in context},
  author={Riehl, Emily},
  year={2017},
  publisher={Courier Dover Publications}
}

@book{johnstone2002sketchesv1,
  title={Sketches of an Elephant: A Topos Theory Compendium, Volume 1},
  author={Johnstone, Peter T.},
  year={2002},
  publisher={Oxford University Press}
}

@book{johnstone2002sketchesv2,
  title={Sketches of an Elephant: A Topos Theory Compendium, Volume 2},
  author={Johnstone, Peter T.},
  % volume={2},
  year={2002},
  publisher={Oxford University Press}
}

@misc{OpenLawvere,
  author = {F. William Lawvere},
  title = {open problems in topos theory},
  % howpublished = "\url{https://www.acsu.buffalo.edu/~wlawvere/Openproblemstopos.htm}",
  url = {https://www.acsu.buffalo.edu/~wlawvere/Openproblemstopos.htm},
  month = jul,
  year = {2016}
}

@article{johnstone1981factorization,
  title={Factorization theorems for geometric morphisms, I},
  author={Johnstone, Peter T.},
  journal={Cahiers de topologie et g{\'e}om{\'e}trie diff{\'e}rentielle cat{\'e}goriques},
  volume={22},
  number={1},
  pages={3--17},
  year={1981}
}

@article{rosenthal1982quotient,
  title={Quotient systems in Grothendieck topoi},
  author={Rosenthal, Kimmo I.},
  journal={Cahiers de topologie et g{\'e}om{\'e}trie diff{\'e}rentielle cat{\'e}goriques},
  volume={23},
  number={4},
  pages={425--438},
  year={1982}
}

@article{rogers2021toposesM,
  title={Toposes of Monoid Actions},
  author={Rogers, Morgan},
  journal={arXiv preprint arXiv:2112.10198},
  year={2021}
}

@article{freyd1980axiom,
  title={The axiom of choice},
  author={Freyd, Peter},
  journal={Journal of Pure and Applied Algebra},
  volume={19},
  pages={103--125},
  year={1980},
  publisher={Elsevier}
}

@book{johnstone2014topos,
  title={Topos theory},
  author={Johnstone, Peter T.},
  year={2014},
  publisher={Courier Corporation}
}

@article{henry2018localic,
  title={The localic isotropy group of a topos},
  author={Henry, Simon},
  journal={Theory and Applications of Categories},
  volume={33},
  number={41},
  pages={1318--1345},
  year={2018}
}

@book{borceux1994handbook,
  title={Handbook of Categorical Algebra: Volume 2, Categories and Structures},
  author={Borceux, Francis},
  % volume={2},
  year={1994},
  publisher={Cambridge University Press}
}

@article{rogers2021toposesT,
  title={Toposes of topological monoid actions},
  author={Rogers, Morgan},
  journal={arXiv preprint arXiv:2105.00772},
  year={2021}
}

@article{khanjanzadeh2021weak,
  title={Weak Topologies on Toposes},
  author={Khanjanzadeh, Zeinab and Madanshekaf, Ali},
  journal={Bulletin of the Iranian Mathematical Society},
  volume={47},
  number={2},
  pages={461--486},
  year={2021},
  publisher={Springer}
}

@article{joyal1981theorie,
  title={Une th{\'e}orie combinatoire des s{\'e}ries formelles},
  author={Joyal, Andr{\'e}},
  journal={Advances in mathematics},
  volume={42},
  number={1},
  pages={1--82},
  year={1981},
  publisher={Elsevier}
}

@article{adamek2001functors,
  title={On functors which are lax epimorphisms},
  author={Adamek, Jiri and El Bashir, Robert and Sobral, Manuela and Velebil, Jiri},
  journal={Theory and Applications of Categories},
  volume={8},
  pages={509--521},
  year={2001},
  publisher={Mount Allison University, Department of Mathematics and Computer Science~…}
}

@article{el2002simultaneously,
  title={Simultaneously reflective and coreflective subcategories of presheaves},
  author={El Bashir, Robert and Velebil, Jiri},
  journal={Theory and Applications of Categories},
  volume={10},
  number={16},
  pages={410--423},
  year={2002}
}
%o%f: 引用を見直す．例えば，Henry Simonのタイトルがバグってるし，アクセプトされた版がある．他のもチェック．(アクセプト版がないか，タイトル等バグってないか)
\end{document}